\documentclass{amsart}

\usepackage[utf8]{inputenc}
\usepackage[all]{xy}
\usepackage{amsmath, amsfonts, amssymb, amscd, tikz, mathrsfs, amsthm}
\usepackage{graphicx}
\usetikzlibrary{arrows}
\usepackage{pgfplots}
\pgfplotsset{compat=1.15}

\newtheorem{definition}{Definition}
\usepackage[colorinlistoftodos]{todonotes}
\usepackage[colorlinks=true, allcolors=blue]{hyperref}
\usepackage[tikz]{}
\usepackage{float}
\usepackage{comment}
\usepackage{cellspace}
\newcommand{\Ext}{\operatorname{Ext}}
\newcommand{\Hom}{\operatorname{Hom}}

\newcommand{\op}{\mathit{op}}
\newcommand{\rep}{\operatorname{rep}}
\newcommand{\udim}{\underline{\operatorname{dim}}}
\newcommand{\Gr}{\operatorname{Gr}}
\newcommand{\univ}{\textup{univ}}
\newcommand{\prin}{\textup{prin}}
\newcommand{\trop}{\textup{trop}}
\newcommand{\ice}{\textup{ice}}

\newcommand{\Po}{\mathcal P}
\newtheorem{lemma}{Lemma}
\newtheorem{example}{Example}
\newtheorem{theorem}{Theorem}
\newtheorem{proposition}{Proposition}

\newtheorem{corollary}{Corollary}
\newcommand{\cc}{{\underline c}}
\newcommand{\uzero}{{\underline 0}}
\newcommand{\dd}{{\underline d}}
\renewcommand{\gg}{{\underline g}}
\newcommand{\ee}{{\underline e}}
\newcommand{\xx}{{\underline x}}
\newcommand{\vv}{{\underline v}}
\newcommand{\pp}{{\underline p}}
\newcommand{\qq}{{\underline q}}

\newcommand{\y}{\underline y}
\renewcommand{\v}{\underline v}
\newcommand{\la}{\langle}
\newcommand{\ra}{\rangle}
\newcommand{\m}{{\underline m}}
\renewcommand{\u}{{\underline u}}
\renewcommand{\b}{{\underline b}}
\renewcommand{\i}{i}

\renewcommand{\v}{\underline v}
\renewcommand{\div}{\operatorname{div}}
\newcommand{\Div}{\operatorname{Div}}
\newcommand{\CDiv}{\operatorname{CDiv}}

\newcommand{\Nef}{\operatorname{Nef}}
\newcommand{\Pic}{\operatorname{Pic}}
\renewcommand{\c}{{\underline c}}

\title[ABHY Associahedra and Newton Polytopes of $F$-polynomials]{ABHY Associahedra and Newton Polytopes of $F$-polynomials for finite type cluster algebras}
\author[Bazier-Matte]
{Véronique Bazier-Matte}
\address{LaCIM, UQAM, Montréal, Québec, Canada}
\email{veronique.b.matte@gmail.com}

\author[Chapelier]
       {Nathan Chapelier-Laget}
       \address{LaCIM, UQAM, Montréal, Québec, Canada}
       \email{nathan.chapelier@gmail.com}

\author[Douville]{Guillaume Douville}
\address{LaCIM, UQAM, Montréal, Québec, Canada}
\email{douvilleg@gmail.com }

\author[Mousavand]{Kaveh Mousavand} 
\address{LaCIM, UQAM, Montréal, Québec, Canada}
\email{mousavand.kaveh@gmail.com }

\author[Thomas]{Hugh Thomas}
\address{LaCIM, UQAM, Montréal, Québec, Canada}
\email{hugh.ross.thomas@gmail.com}

\author[Yıldırım]{Emine Yıldırım}
\address{LaCIM, UQAM, Montréal, Québec, Canada}
\email{emineyyildirim@gmail.com }

\thanks{VBM and GD were supported by NSERC Alexander Graham Bell graduate scholarships.  KM and EY were partially supported by ISM scholarships. HT was supported by the Canada Research Chairs program and an NSERC Discovery Grant.}

\subjclass[2010]{13F60,16G20}

\begin{document}

\begin{abstract} A new construction of the associahedron was recently given by Arkani-Hamed, Bai, He, and Yan in connection with the physics of scattering amplitudes.  We show that their construction (suitably understood) can be applied to construct generalized associahedra of any simply-laced Dynkin type.  Unexpectedly, we also show that this same construction produces Newton polytopes for all the $F$-polynomials of the corresponding cluster algebras. In addition, we show that the toric variety associated to the $g$-vector fan has the property that its nef cone is simplicial.
\end{abstract}
\maketitle

\section{Introduction}
Let $Q$ be a Dynkin quiver.  That is to say, $Q$ is an
orientation of a simply-laced Dynkin diagram, with vertices numbered 1 to $n$.
Let $B_0$ be the matrix with the property that the entry
$(B_0)_{ij}$ equals the number of arrows from $i$ to $j$ minus the number
of arrows from $j$ to $i$.  Starting from the matrix $B_0$, one
can define a corresponding cluster algebra $\mathcal A(Q)$, which is a
commutative ring
with a distinguished set of generators, known as cluster variables, grouped together into overlapping sets of size $n$ known as clusters.
Because
of our particular choice of $Q$, there are only finitely many cluster variables.
Each cluster variable has an associated vector in $\mathbb Z^n$, its
$g$-vector.  The $g$-vectors are the rays of the $g$-vector fan, the
maximal-dimensional cones of which correspond to clusters.  This fan is proper, in the sense that the union of its cones is all of $\mathbb R^n$.

Considerable attention has been given to the problem of constructing polytopes whose outer normal fan is the $g$-vector fan.  Such polytopes are called \emph{generalized associahedra}.  
The question of whether the $g$-vector fan can be realized in this way was
first raised by Fomin and Zelevinsky in \cite{FZ}, and first solved by Chapoton, Fomin, and Zelevinsky \cite{CFZ}.  In fact,
$g$-vectors had not yet been defined at the time of these two papers, but
the fan which they study would subsequently be recognized as the
$g$-vector fan associated to a particular orientation of each Dynkin diagram.  Subsequently,
\cite{HLT} gave a construction which solves the problem as described here.  \cite{HPS} solves a more general problem, where the initial quiver is assumed only to be mutation-equivalent to a Dynkin quiver, and not necessarily Dynkin itself.  The papers we have cited actually work in greater generality, in that they also treat Dynkin types which are not simply laced.  In this paper we focus on the simply-laced case because it is technically easier.

The prehistory of this problem goes back much further.  The combinatorics of the face lattice of a
generalized associahedron is not sensitive to the orientation of $Q$.  When
$Q$ is of type $A_n$, the face lattice is that of the associahedron, as originally defined, as a cell complex, by Stasheff \cite{Sta}.  The first polytopal realization to appear in the literature is due to Lee \cite{Lee}, and many others followed.  An excellent overview is provided by \cite{CSZ}.  
One particular realization of the associahedron has been
studied by many authors \cite{SS,Loday,Postnikov,RSS}, including Loday, whose name is most often associated to it.
The associahedra constructed in this way turn out to be generalized associahedra corresponding to the linear orientation of the
$A_n$ diagram.  Recently, yet another construction of this associahedron was
given in the physics literature by Arkani-Hamed, Bai, He, and Yan \cite{ABHY}, in connection with
scattering amplitudes for bi-adjoint scalar $\varphi^3$  
theory.  
We refer to this as the ABHY construction.

In this paper, we extend the ABHY construction to arbitrary (simply-laced)
Dynkin quivers.  We further show that, quite surprisingly, this same construction realizes the Newton polytopes of the $F$-polynomials of the corresponding cluster algebras.  
(The $F$-polynomials are certain polynomials in $n$ variables which are a
reparameterization of the cluster variables; in particular, the cluster
variables can be recovered from them.)  In fact, the ABHY
construction can be seen even more naturally as constructing Newton polytopes of
certain
{universal $F$-polynomials}, which we define.
The universal $F$-polynomials bear the
same relationship to the cluster algebra with universal coefficients as the
usual $F$-polynomials do to the cluster algebra with principal coefficients.
We also consider the toric variety associated to the $g$-vector fan. We show that our results on realizations of associahedra imply that the nef cone of this toric variety is simplicial.

After the appearance of the first version of this paper in 2018,
  Palu, Padrol, Pilaud, and Plamondon \cite{PPPP} showed that a very similar construction can be applied to construction of
  generalized associahedra for all seeds in all Dynkin type cluster algebras.
  The first author of the present paper also extended the techniques in the present paper in her thesis \cite{BM}, giving another proof that the same construction can be applied to any seed of simply-laced Dynkin type, and that this construction also gives Newton polytopes of $F$-polynomials in the same generality.
  Fei \cite{F1,F2} takes a less explicit approach but proves a very general result, to the effect that if $A$ is any algebra with finitely many $\tau$-rigid
  indecomposable modules, then a polytope dual to the $\tau$ tilting fan of $A$ can be constructed as the Minkowski sum of the submodule polytopes of the
  $\tau$-rigid indecomposables.

  Our result that our polytopes also yield Newton polytopes of $F$-polynomials
  has subsequently been extended to non-simply laced types by Arkani-Hamed, He,
and Lam \cite{AHL}, by using a folding argument to reduce to the case we consider.

\section{Construction}
In the interests of self-containedness, we will begin with a completely
explicit, if somewhat unmotivated, description of our construction.  We will
then provide a more representation-theoretic description, which is needed
for the proof of correctness.  Statements which are not proved in this section will be proved in the following section (or will turn out to be equivalent to well-known facts from the theory of quiver representations).


We write $Q_0$ for the set of vertices of $Q$ and $Q_1$ for the set of arrows.
We assume that $Q_0=\{1,\dots,n\}$.
We write
$Q^\op$ for the opposite quiver of $Q$, all of whose arrows are reversed compared to
$Q$.

Draw $\mathbb Z_{\geq 0}$ many copies of $Q$.  The vertices in this quiver are denoted
$(i,j)$ where $i \in \mathbb Z_{\geq 0}$ and $j\in Q_0$.  We also add
arrows between the copies of $Q$: if there is an arrow from $j$ to $k$
in $Q$, we put an arrow from $(i,k)$ to $(i+1,j)$.  This infinite quiver
we denote by $\mathbb Z_{\geq 0}Q$.  
See Example \ref{ex-a3} below for an example of the initial part of the 
quiver $\mathbb Z_{\geq 0}Q$ for $Q$ the quiver $\xymatrix{1 \ar[r] & 2   & 3 \ar[l]}$.

We associate to each vertex $(i,j)$ a vector in
$\mathbb Z^n$, which we call the dimension vector, and which we denote $\udim(i,j)$.  (What exactly it is the dimension of will be explained in the following section.  For now, it is simply an integer vector.)
To $(0,j)$, we associate the dimension vector $\udim(0,j)$
obtained by
putting a 1 at every vertex that can be reached from $j$ by following arrows
of $Q^\op$ (including the vertex $j$ itself), and 0 at all other vertices. For $(i,j)$ with $i>0$, we associate
the dimension vector which satisfies:
$$\udim(i,j)+\udim(i-1,j)=\sum_{(i-1,j)\rightarrow (i',j')\rightarrow (i,j)} \udim(i',j')$$

Here, the sum on the righthand side runs over all vertices $(i',j')$
on a path of
length two from $(i-1,j)$ to $(i,j)$.  Starting with the dimension vectors already defined
for $(0,j)$, these equations allow us to deduce the value of $\udim(i,j)$ for all
$(i,j)$ in $\mathbb Z_{\geq 0}Q$ inductively.   

It turns out that the dimension vectors calculated in this way have the
property that they are non-zero and sign-coherent, in the sense that each
$\udim(i,j)$ either has all entries non-negative or all entries non-positive.
In these two cases we simply say
that $\udim(i,j)$ is non-negative or non-positive,
respectively.  

A certain subset of the vertices of $\mathbb Z_{\geq 0}Q$ are in natural correspondennce with the cluster variables of $\mathcal A(Q)$.  For $1\leq j\leq n$,
define $i_j$ to be
the maximal index such that all of $\udim(0,j),\udim(1,j),\dots,\udim(i_j,j)$
are non-negative.  

Define
\begin{eqnarray*}
  \mathcal I^+&=&\{(i,j) \mid 1\leq j \leq n, 0\leq i \leq i_j\}\\
  \mathcal I&=&\{(i,j) \mid 1\leq j \leq n, 0\leq i \leq i_j+1\}
  \end{eqnarray*}

There is a natural bijection between the elements of $\mathcal I$ and the cluster variables of $A(Q)$.  We recall the details of the correspondence in Section \ref{cluster-intro} below.  We write $x_{ij}$ for the cluster variable associated to $(i,j)\in\mathcal I$.  We say that two cluster variables are \emph{compatible} if there is some cluster which contains both of them.  We say that two elements of $\mathcal I$ are compatible if the corresponding cluster variables are.  

Consider a real vector space with basis indexed by the elements of $\mathcal I$, say $V=\mathbb R^{\mathcal I}$.  For $(i,j)\in \mathcal I$, we write $p_{ij}$ for the coordinate function on $V$ indexed by $(i,j)$. 

Fix a collection of positive integers $\cc=(c_{ij})_{ij \in \mathcal I^+}$.
We construct an associahedron for each choice of $\cc$.

Consider the following collection of equations,
one for each $(i,j)\in \mathcal I^+$.

$$p_{i,j} + p_{i+1,j} = c_{ij} + \sum_{(i,j) \rightarrow (i',j') \rightarrow (i+1,j)} p_{i',j'}$$

We call these equations the \emph{$\cc$-deformed mesh relations}.
They define an $n$-dimensional affine space $\mathbb E_\cc$
inside $V$.

Write $\mathbb U_\cc$ for the region inside $\mathbb E_\cc$ all of whose coordinates
$p_{ij}$ are non-negative.  That is to say, $\mathbb U_\cc$ is the intersection of
the positive orthant in $V$ with $\mathbb E_\cc$.

\begin{example}~\label{ex-a3}
  Let us consider the quiver $Q:\xymatrix{1 \ar[r] & 2   & 3 \ar[l]}$. The following shows the part of $\mathbb Z_{\geq 0}Q$ whose vertices
  are in $\mathcal{I}$, labelled by the elements of $\mathcal I$.  

\[\begin{tikzpicture}[xscale=1.3, yscale=1]
    \draw (-3,2.4) node (01) {$(0,1)$};
    \draw (-3,0) node (02)   {$(0,3)$};
    \draw (-2,1.2) node (03) {$(0,2)$};
    \draw (-1,2.4) node (11) {$(1,1)$};
    \draw (-1,0) node (12)   {$(1,3)$};
    \draw (0,1.2) node (13)  {$(1,2)$};
    \draw (1,2.4) node (21)  {$(2,1)$};
    \draw (1,0) node  (22)   {$(2,3)$};
    \draw (2,1.2) node (23)  {$(2,2)$};
    \draw [->] (01) -- (03);
    \draw [->] (02) -- (03);
    \draw [->] (03) -- (11);
    \draw [->] (03) -- (12);
    \draw [->] (11) -- (13);
    \draw [->] (12) -- (13);
    \draw [->] (13) -- (21);
    \draw [->] (13) -- (22);
    \draw [->] (21) -- (23);
    \draw [->] (22) -- (23);
\end{tikzpicture} \]

The corresponding dimension vectors are:

\[\begin{tikzpicture}[xscale=1.3, yscale=1]
    \draw (-3,2.4) node (01) {$(1,0,0)$};
    \draw (-3,0) node (02)   {$(0,0,1)$};
    \draw (-2,1.2) node (03) {$(1,1,1)$};
    \draw (-1,2.4) node (11) {$(0,1,1)$};
    \draw (-1,0) node (12)   {$(1,1,0)$};
    \draw (0,1.2) node (13)  {$(0,1,0)$};
    \draw (1,2.4) node (21)  {$(0,0,-1)$};
    \draw (1,0) node  (22)   {$(-1,0,0)$};
    \draw (2,1.2) node (23)  {$(-1,-1,-1)$};
    \draw [->] (01) -- (03);
    \draw [->] (02) -- (03);
    \draw [->] (03) -- (11);
    \draw [->] (03) -- (12);
    \draw [->] (11) -- (13);
    \draw [->] (12) -- (13);
    \draw [->] (13) -- (21);
    \draw [->] (13) -- (22);
    \draw [->] (21) -- (23);
    \draw [->] (22) -- (23);
\end{tikzpicture} \]

We fix a six-tuple of positive integers $\cc=(c_{ij})_{(i,j)\in \mathcal I^+}$.  
The region $\mathbb E_\cc$ is cut out by the following equations:
\begin{align*}
p_{01}+p_{11}&=p_{02}+c_{01}\\
p_{03}+p_{13}&=p_{02}+c_{03}\\
p_{02}+p_{12}&=p_{11}+p_{13}+c_{02}\\
p_{11}+p_{21}&=p_{12}+c_{11}\\
p_{13}+p_{23}&=p_{12}+c_{13}\\
p_{12}+p_{22}&=p_{21}+p_{23}+c_{12}
\end{align*}
\end{example}

Notice that, by construction, the dimension vectors 
satisfy the $\underline 0$-deformed mesh relations (which are generally called the \emph{mesh relations}), that is to say, the deformed mesh relations with the deformation parameters set to zero.  
There is another important collection of vectors which 
satisfy these equations: the \emph{$g$-vectors}.  By 
definition, $g(0,j)$ is the $j$-th standard basis vector,
and the other $g$-vectors are determined by the
mesh relations.  The \emph{$g$-vector fan} is the fan whose rays are the $g$-vectors, and such that a collection of rays generates a cone of the fan if and only if the corresponding collection of cluster variables is compatible.   

\begin{example}\label{ex-a3-b}
We continue Example~\ref{ex-a3}. In this case the corresponding $g$-vectors are as follows:
  
\[\begin{tikzpicture}[xscale=1.3, yscale=1]
    \draw (-3,2.4) node (01) {$(1,0,0)$};
    \draw (-3,0) node (02)   {$(0,0,1)$};
    \draw (-2,1.2) node (03) {$(0,1,0)$};
    \draw (-1,2.4) node (11) {$(-1,1,0)$};
    \draw (-1,0) node (12)   {$(0,1,-1)$};
    \draw (0,1.2) node (13)  {$(-1,1,-1)$};
    \draw (1,2.4) node (21)  {$(0,0,-1)$};
    \draw (1,0) node  (22)   {$(-1,0,0)$};
    \draw (2,1.2) node (23)  {$(0,-1,0)$};
    \draw [->] (01) -- (03);
    \draw [->] (02) -- (03);
    \draw [->] (03) -- (11);
    \draw [->] (03) -- (12);
    \draw [->] (11) -- (13);
    \draw [->] (12) -- (13);
    \draw [->] (13) -- (21);
    \draw [->] (13) -- (22);
    \draw [->] (21) -- (23);
    \draw [->] (22) -- (23);
\end{tikzpicture} \]

\end{example}

Note that, in the example, the $g$-vectors corresponding to the elements of
$\mathcal I\setminus \mathcal I^+$ are the negative standard basis
vectors.  This is a general phenomenon, and allows us to define an important projection $\pi:V\rightarrow \mathbb R^n$: the $k$-th
coordinate of the projection to $\mathbb R^n$ is given by $p_{ij}$ for $(i,j)\in \mathcal I\setminus \mathcal I^+$ such that $g(i,j)=-e_k$.
This projection defines a bijection between $\mathbb E_\cc$ and $\mathbb R^n$.
We define $\mathbb A_\cc=\pi(\mathbb U_\cc)$. 

Given a full-dimensional polytope $P$ in $\mathbb R^n$, there is fan associated to it called the \emph{outer normal fan} of $P$, denoted $\Sigma_P$.  For each facet $F$ of $P$, let $\rho_F$ be the ray pointing in the direction perpendicular to $F$ and away from $P$.  The collection of rays $\{\rho_F \mid F$ is a facet of $P\}$ are the rays of $\Sigma_P$; the set of rays 
$\{\rho_{F_1},\dots,\rho_{F_j}\}$ generates a cone in $\Sigma_P$ if and only if there is a face $G$ of $P$ such that the facets of $P$ containing $G$ are exactly $F_1,\dots,F_j$.

The following is our main theorem about realizing 
generalized associahedra.  

\begin{theorem}\label{th-one} 
\begin{enumerate}
\item Each facet of $\mathbb U_\cc$ is defined by the vanishing of exactly one
coordinate of $V$. 
\item If $G$ is a face of $\mathbb U_\cc$, then the set 
$\{ \alpha\in \mathcal I \mid G$ lies on the
hyperplane $p_\alpha=0 \}$ is a compatible subset of 
$\mathcal I$, and every compatible subset arises this way.
\item In particular, the vertices of $\mathbb U_\cc$ 
correspond to clusters.
\item $\pi$ is an isomorphism of affine spaces between
$\mathbb E_\cc$ and $\mathbb R^n$.  
Consequently, the faces of $\mathbb A_\cc$ also 
correspond bijectively to compatible sets in $\mathcal I$.  
\item If $F_{ij}$ is the facet of $\mathbb U_\cc$ given by $p_{ij}=0$, the normal to
$\pi(F_{ij})$ oriented away from $\mathbb A_\cc$ is the ray 
generated by $g(i,j)$.   
\end{enumerate}\end{theorem}

\begin{example}\label{ex-a3-c} We continue Example~\ref{ex-a3-b}.

Here we show an illustration of $\mathbb A_\cc$ in this case.  
At each vertex, we have drawn a small copy of
the quiver with vertex set $\mathcal I$ (with the arrows omitted) on which
we have marked the cluster corresponding to the vertex.  

\begin{figure}[H]
\includegraphics[width=7cm]{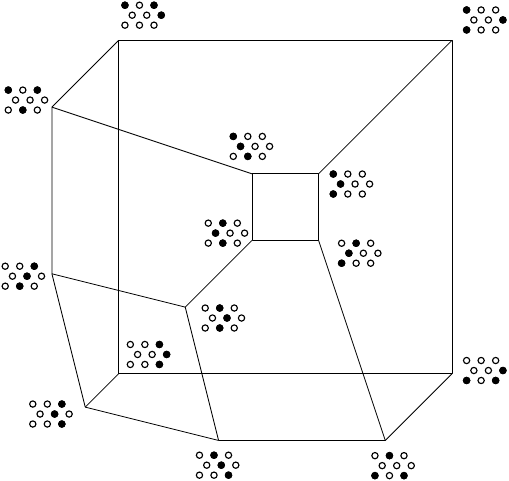}
\caption{Associahedron corresponding to $1 \rightarrow 2 \leftarrow3$.}
\end{figure}

To represent the vertices of $\mathbb U_\cc$ in $V$, it is convenient to write the values of $p_{ij}$ at position $(i,j)$, with the value $c_{ij}$ positioned between $(i,j)$ and $(i+1,j)$.  We write $c_{ij}$ in red
to make clear the distinction between the two kinds of entries.  
To see $\mathbb A_\cc$, we forget everything in each row except the last
entry.  

In the above example, we set all the $c_{ij}=1$.  The vertices 
are as follows:
\begin{align*} \begin{tikzpicture}[xscale=.5,yscale=.5]
\node (b) at (-1,1) {0};
\node (c) at (-1,-1) {0};
\node (a) at (0,0) {0};
\node (d) at (1,1) {1};
\node (e) at (1,-1) {1};
\node (f) at (2,0) {3};
\node (g) at (3,1) {3};
\node (h) at (3,-1) {3 };
\node (i) at (4,0) {4};
\node[red] (p) at (0,1) {1};
\node[red] (q) at (0,-1) {1};
\node[red] (r) at (1,0) {1};
\node[red] (s) at (2,1) {1};
\node[red] (t) at (2,-1) {1};
\node[red]  (u) at (3,0) {1};
\end{tikzpicture}
\quad
\begin{tikzpicture}[xscale=.5,yscale=.5]
\node (b) at (-1,1) {0};
\node (c) at (-1,-1) {1};
\node (a) at (0,0) {0};
\node (d) at (1,1) {1};
\node (e) at (1,-1) {0};
\node (f) at (2,0) {2};
\node (g) at (3,1) {2};
\node (h) at (3,-1) {3};
\node (i) at (4,0) {4};
\node[red] (p) at (0,1) {1};
\node[red] (q) at (0,-1) {1};
\node[red] (r) at (1,0) {1};
\node[red] (s) at (2,1) {1};
\node[red] (t) at (2,-1) {1};
\node[red]  (u) at (3,0) {1};
\end{tikzpicture}
\quad
\begin{tikzpicture}[xscale=.5,yscale=.5]
\node (b) at (-1,1) {1};
\node (c) at (-1,-1) {0};
\node (a) at (0,0) {0};
\node (d) at (1,1) {0};
\node (e) at (1,-1) {1};
\node (f) at (2,0) {2};
\node (g) at (3,1) {3};
\node (h) at (3,-1) {2};
\node (i) at (4,0) {4};
\node[red] (p) at (0,1) {1};
\node[red] (q) at (0,-1) {1};
\node[red] (r) at (1,0) {1};
\node[red] (s) at (2,1) {1};
\node[red] (t) at (2,-1) {1};
\node[red]  (u) at (3,0) {1};
\end{tikzpicture}\\
\begin{tikzpicture}[xscale=.5,yscale=.5]
\node (b) at (-1,1) {1};
\node (c) at (-1,-1) {1};
\node (a) at (0,0) {0};
\node (d) at (1,1) {0};
\node (e) at (1,-1) {0};
\node (f) at (2,0) {1};
\node (g) at (3,1) {2};
\node (h) at (3,-1) {2};
\node (i) at (4,0) {4};
\node[red] (p) at (0,1) {1};
\node[red] (q) at (0,-1) {1};
\node[red] (r) at (1,0) {1};
\node[red] (s) at (2,1) {1};
\node[red] (t) at (2,-1) {1};
\node[red]  (u) at (3,0) {1};
\end{tikzpicture}
\quad
\begin{tikzpicture}[xscale=.5,yscale=.5]
\node (b) at (-1,1) {2};
\node (c) at (-1,-1) {2};
\node (a) at (0,0) {1};
\node (d) at (1,1) {0};
\node (e) at (1,-1) {0};
\node (f) at (2,0) {0};
\node (g) at (3,1) {1};
\node (h) at (3,-1) {1};
\node (i) at (4,0) {3};
\node[red] (p) at (0,1) {1};
\node[red] (q) at (0,-1) {1};
\node[red] (r) at (1,0) {1};
\node[red] (s) at (2,1) {1};
\node[red] (t) at (2,-1) {1};
\node[red]  (u) at (3,0) {1};
\end{tikzpicture}
\quad
\begin{tikzpicture}[xscale=.5,yscale=.5]
\node (b) at (-1,1) {2};
\node (c) at (-1,-1) {3};
\node (a) at (0,0) {2};
\node (d) at (1,1) {1};
\node (e) at (1,-1) {0};
\node (f) at (2,0) {0};
\node (g) at (3,1) {0};
\node (h) at (3,-1) {1};
\node (i) at (4,0) {2};
\node[red] (p) at (0,1) {1};
\node[red] (q) at (0,-1) {1};
\node[red] (r) at (1,0) {1};
\node[red] (s) at (2,1) {1};
\node[red] (t) at (2,-1) {1};
\node[red]  (u) at (3,0) {1};
\end{tikzpicture}\\
\begin{tikzpicture}[xscale=.5,yscale=.5]
\node (b) at (-1,1) {3};
\node (c) at (-1,-1) {2};
\node (a) at (0,0) {2};
\node (d) at (1,1) {0};
\node (e) at (1,-1) {1};
\node (f) at (2,0) {0};
\node (g) at (3,1) {1};
\node (h) at (3,-1) {0};
\node (i) at (4,0) {2};
\node[red] (p) at (0,1) {1};
\node[red] (q) at (0,-1) {1};
\node[red] (r) at (1,0) {1};
\node[red] (s) at (2,1) {1};
\node[red] (t) at (2,-1) {1};
\node[red]  (u) at (3,0) {1};
\end{tikzpicture}
\quad
\begin{tikzpicture}[xscale=.5,yscale=.5]
\node (b) at (-1,1) {3};
\node (c) at (-1,-1) {3};
\node (a) at (0,0) {3};
\node (d) at (1,1) {1};
\node (e) at (1,-1) {1};
\node (f) at (2,0) {0};
\node (g) at (3,1) {0};
\node (h) at (3,-1) {0};
\node (i) at (4,0) {1};
\node[red] (p) at (0,1) {1};
\node[red] (q) at (0,-1) {1};
\node[red] (r) at (1,0) {1};
\node[red] (s) at (2,1) {1};
\node[red] (t) at (2,-1) {1};
\node[red]  (u) at (3,0) {1};
\end{tikzpicture}
\quad
\begin{tikzpicture}[xscale=.5,yscale=.5]
\node (b) at (-1,1) {3};
\node (c) at (-1,-1) {3};
\node (a) at (0,0) {4};
\node (d) at (1,1) {2};
\node (e) at (1,-1) {2};
\node (f) at (2,0) {1};
\node (g) at (3,1) {0};
\node (h) at (3,-1) {0};
\node (i) at (4,0) {0};
\node[red] (p) at (0,1) {1};
\node[red] (q) at (0,-1) {1};
\node[red] (r) at (1,0) {1};
\node[red] (s) at (2,1) {1};
\node[red] (t) at (2,-1) {1};
\node[red]  (u) at (3,0) {1};
\end{tikzpicture}
\\
\begin{tikzpicture}[xscale=.5,yscale=.5]
\node (b) at (-1,1) {3};
\node (c) at (-1,-1) {0};
\node (a) at (0,0) {2};
\node (d) at (1,1) {0};
\node (e) at (1,-1) {3};
\node (f) at (2,0) {2};
\node (g) at (3,1) {3};
\node (h) at (3,-1) {0};
\node (i) at (4,0) {2};
\node[red] (p) at (0,1) {1};
\node[red] (q) at (0,-1) {1};
\node[red] (r) at (1,0) {1};
\node[red] (s) at (2,1) {1};
\node[red] (t) at (2,-1) {1};
\node[red]  (u) at (3,0) {1};
\end{tikzpicture}
\quad
\begin{tikzpicture}[xscale=.5,yscale=.5]
\node (b) at (-1,1) {0};
\node (c) at (-1,-1) {3};
\node (a) at (0,0) {2};
\node (d) at (1,1) {3};
\node (e) at (1,-1) {0};
\node (f) at (2,0) {2};
\node (g) at (3,1) {0};
\node (h) at (3,-1) {3};
\node (i) at (4,0) {2};
\node[red] (p) at (0,1) {1};
\node[red] (q) at (0,-1) {1};
\node[red] (r) at (1,0) {1};
\node[red] (s) at (2,1) {1};
\node[red] (t) at (2,-1) {1};
\node[red]  (u) at (3,0) {1};
\end{tikzpicture}
\quad
\begin{tikzpicture}[xscale=.5,yscale=.5]
\node (b) at (-1,1) {3};
\node (c) at (-1,-1) {0};
\node (a) at (0,0) {4};
\node (d) at (1,1) {2};
\node (e) at (1,-1) {5};
\node (f) at (2,0) {4};
\node (g) at (3,1) {3};
\node (h) at (3,-1) {0};
\node (i) at (4,0) {0};
\node[red] (p) at (0,1) {1};
\node[red] (q) at (0,-1) {1};
\node[red] (r) at (1,0) {1};
\node[red] (s) at (2,1) {1};
\node[red] (t) at (2,-1) {1};
\node[red]  (u) at (3,0) {1};
\end{tikzpicture}
\\
\begin{tikzpicture}[xscale=.5,yscale=.5]
\node (b) at (-1,1) {0};
\node (c) at (-1,-1) {3};
\node (a) at (0,0) {4};
\node (d) at (1,1) {5};
\node (e) at (1,-1) {2};
\node (f) at (2,0) {4};
\node (g) at (3,1) {0};
\node (h) at (3,-1) {3};
\node (i) at (4,0) {0};
\node[red] (p) at (0,1) {1};
\node[red] (q) at (0,-1) {1};
\node[red] (r) at (1,0) {1};
\node[red] (s) at (2,1) {1};
\node[red] (t) at (2,-1) {1};
\node[red]  (u) at (3,0) {1};
\end{tikzpicture}
\quad
\begin{tikzpicture}[xscale=.5,yscale=.5]
\node (b) at (-1,1) {0};
\node (c) at (-1,-1) {0};
\node (a) at (0,0) {4};
\node (d) at (1,1) {5};
\node (e) at (1,-1) {5};
\node (f) at (2,0) {7};
\node (g) at (3,1) {3};
\node (h) at (3,-1) {3};
\node (i) at (4,0) {0};
\node[red] (p) at (0,1) {1};
\node[red] (q) at (0,-1) {1};
\node[red] (r) at (1,0) {1};
\node[red] (s) at (2,1) {1};
\node[red] (t) at (2,-1) {1};
\node[red]  (u) at (3,0) {1};
\end{tikzpicture}
\qquad\qquad\quad
\end{align*}
\end{example}

\section{Proof of correctness}\label{section-three}

Since the foundational work of \cite{MRZ,BMRRT}, it has been clear that
representations of quivers are extremely useful in understanding the
combinatorics of cluster algebras.  We will use a setting that is inspired
by \cite{BMRRT}, though described in somewhat different terms.

The quiver $\mathbb Z_{\geq 0}Q$ restricted to the vertices of $
\mathcal I^+$, gives the Auslander-Reiten quiver for the category of
representations of $Q^\op$.  If we restrict to $\mathcal I$ instead, we get
$n$ additional vertices added to the righthand end.  We understand this as
the Auslander-Reiten quiver of a full subcategory of $D^b(\rep Q^\op)$,
whose vertices correspond to the indecomposable quiver representations
together with the $P_j[1]$ where $P_j$ is the projective representation at
vertex $j$ and $[1]$ is the shift functor.  We write $W_{ij}$ for the object in $D^b(\rep Q^\op)$ corresponding
to the vertex $(i,j)\in \mathcal I$.
The indecomposable projective $P_j$ is $W_{0j}$.
As is well known, $\udim(i,j)$, as defined in the previous section, is the dimension vector of $W_{ij}$ if $(i,j)\in \mathcal I^+$; to include the cases of $(i,j)\in \mathcal I\setminus \mathcal I^+$ as well, we can say that $\udim(i,j)$ is the class in the Grothendieck group of $D^b(\rep Q^\op)$ of $W_{ij}$.  

Because it will be useful later in the paper, we will
begin by weakening the hypothesis on the tuple $\cc$: we 
will begin by assuming that its entries are non-negative, rather than all being
strictly positive.  
Let
$$M_\cc=\bigoplus_{(i,j)\in \mathcal I^+}W_{ij}^{\oplus c_{ij}}$$

Representation theory gives us a natural point $\vv_\cc$ in $V$.  The coordinates of this point are defined by:
$$v_{ij}= \dim\Hom(W_{ij},M_\cc).$$

\begin{lemma} The point $\vv_\cc$ is in $\mathbb E_\cc$. \end{lemma}

\begin{proof}  Suppose that $(i,j)\in\mathcal I^+$.
Let $E_{ij}$ be the direct sum of all the $W_{i'j'}$ for $(i',j')$ on a path of length two between 
$(i,j)$ and $(i+1,j)$.  
We therefore have an Auslander-Reiten triangle in $D^b(\rep Q)$:
$$ W_{ij} \rightarrow E_{ij} \rightarrow W_{i+1,j} \rightarrow
W_{ij}[1]$$
We must verify, for each $(i,j)\in \mathcal I^+$ that 
\begin{equation}\label{eq}
\dim(\Hom(W_{ij},M_\cc)) -\dim(\Hom(E_{ij},M_\cc)) + \dim(\Hom(W_{i+1,j},M_\cc)) = c_{ij}
\end{equation}
We know that the following sequence is exact except at the righthand end:
$$0=\Hom(W_{ij}[1],M_\cc) \rightarrow \Hom(W_{i+1,j},M_\cc) \rightarrow \Hom(E_{ij},M_\cc) \rightarrow \Hom(W_{ij},M_\cc)$$
Thus, the lefthand side of (\ref{eq}) is nothing but the dimension of the quotient of $\Hom(W_{ij},M_\cc)$ by the 
image of $\Hom(E_{ij},M_\cc)$.  By the definition of Auslander-Reiten triangles, any map from $W_{ij}$ to a summand
of $M_\cc$ which is not an isomorphism factors through $E_{ij}$.  The lefthand side of (\ref{eq}) is therefore the
multiplicity of $W_{ij}$ in $M_\cc$, which is exactly $c_{ij}$.  
\end{proof}
We can now say $$\mathbb E_c=\vv_\cc + \mathbb E_\uzero.$$

Here, by $\mathbb E_\uzero$, we mean the points of $V$ satisfying the (\uzero-deformed) mesh relations.  We see that $\mathbb E_\uzero$ is an $n$-dimensional vector space.  For $1\leq k \leq n$, define a vector $\dd^k=(d^k_{ij})_{ij\in\mathcal I}$ by setting $d^k_{ij}=\udim(i,j)_k$.  Then $\dd^1,\dots,\dd^n$ are a basis for
$\mathbb E_\uzero$.  

Similarly, for $1\leq k\leq n$, and $(i,j)\in\mathcal I$, let 
$g^k_{ij}$ be the signed multiplicity of $P_k$ in a projective resolution for
$W_{ij}$.  That is to say, we take a resolution
\begin{equation}\label{eq2}
  P^1 \rightarrow P^0\rightarrow W_{ij} \rightarrow P^1[1] \end{equation}
and take $g^k_{ij}$ to be the multiplicity of $P_k$ in $P^0$ minus
its multiplicity in $P^1$.  Now define the vector $\gg^k=(g^k_{ij})_{ij\in\mathcal I}$.  
Then the set of $\gg^k$ are also a basis for $\mathbb E_0$.  This is clear because (\ref{eq2}) implies that
$\udim W_{ij}=\udim P^0-\udim P^1$, from which we see that we can calculate
the $g$-vector of $W_{ij}$ by expressing $\udim W_{ij}$ in the basis
$\udim P_1, \udim P_2, \dots, \udim P_n$ of $\mathbb Z^n$.  It follows that the collection of vectors
$\{\gg^k\}$ are related to the vectors $\{\dd^k\}$ by a change of basis. 

As we have already mentioned, there is a relation of compatibilty of cluster variables, which we can take as defining a notion of compatibility for elements of $\mathcal I$.  Two elements $\alpha,\beta$ of $\mathcal I$ are
compatible iff $\Ext^1(W_\alpha,W_\beta)=0=\Ext^1(W_\beta,W_\alpha)$. See ~\cite[Eq. (3.3)]{MRZ} and 
\cite[Corollary 4.3]{BMRRT}.

The following is the key lemma.

\begin{lemma} \label{key} 
Suppose the tuple $\cc$ consists of strictly positive 
integers.  Let $\alpha,\beta\in\mathcal I$.  If they are incompatible,
  then there is no point of $\mathbb U_\cc$ lying on the intersection of the hyperplanes $p_\alpha=0$ and $p_\beta=0$.\end{lemma}

\begin{proof} If $\alpha$ and $\beta$ are incompatible, then
  $\Ext^1(W_\alpha,W_\beta)\ne 0$ or
$\Ext^1(W_\beta,W_\alpha)\ne 0$.  Without loss of generality, suppose the former.  This implies, in particular, that $\beta \in \mathcal{I}^+$. Choose a non-split triangle $$W_\beta \rightarrow E \rightarrow W_\alpha \rightarrow W_\beta[1]$$.

Suppose that $\xx=(x_\gamma)_{\gamma\in\mathcal I} \in\mathbb E_\cc$.  Since $\mathbb E_\cc=\vv_\cc+\mathbb E_\uzero$, and 
$\dd^1,\dots,\dd^n$ span $\mathbb E_\uzero$, there
is some $n$-tuple $(m_1,\dots,m_n)$ such that $\xx=\vv_\cc+\sum_{k=1}^n m_k\dd^k$ or in other words, for all $\gamma\in \mathcal I$,
we have $$x_\gamma= \dim\Hom(W_\gamma,M_\cc)+\sum_{k=1}^n m_k \udim(\gamma)_k $$

Now suppose that $x_\alpha=x_\beta=0$.  Note that $\dim E_i=\dim (W_\beta)_i + \dim (W_\alpha)_i$.  Also note that we have the Hom long exact sequence

$$\Ext^{-1}(W_\beta,M_\cc) \rightarrow \Hom(W_\alpha,M_\cc) \rightarrow \Hom(E,M_\cc) \rightarrow \Hom(W_\beta,M_\cc)
\rightarrow \Ext^1(W_\alpha,M_\cc)$$

On the lefthand end $\Ext^{-1}(W_\beta,M_\cc)=\Hom(W_\beta,M_\cc[-1])=0$.

Since $c_\beta>0$, the map from $\Hom(W_\beta,M_\cc)$ to $\Ext^1(W_\alpha,M_\cc)$ is non-zero.
Thus, $\dim \Hom(E,M_\cc)< \dim \Hom(W_\alpha,M_\cc)+\dim\Hom(W_\beta,M_\cc)$.

Therefore

\begin{eqnarray*}&&\dim\Hom(E,M_\cc)+\sum_{k=1}^n m_k\dim E_k
\\ &&\qquad < \dim\Hom(W_\alpha,M_\cc)+\dim\Hom(W_\beta,M_\cc) + \sum_{k=1}^n m_k (\dim (W_\alpha)_k+\dim (W_\beta)_k)
\\ &&\qquad = x_\alpha + x_\beta= 0 \end{eqnarray*}

But the quantity on the lefthand side is a just a sum of the coordinates of
$\xx$ evaluated at the summands of $E$, weighted by their multiplicities.  Thus at least one of the coordinates of $\xx$ is negative, so $\xx$ is not in
$\mathbb U_\cc$.\end{proof}

\begin{lemma}\label{bounded} $\mathbb U_c$ is bounded. \end{lemma}


\begin{proof}
If $\xx$ is in $\mathbb U_\cc$ then, as in the previous proof, there is an
$n$-tuple $(m'_1,\dots,m'_n)$ such that $$\xx=\vv_\cc+\sum_{k=1}^n
m'_k g^k.$$  

Consider what this equation says at some coordinate 
$\alpha\in\mathcal I\setminus \mathcal I^+$.  
Note first that $(\vv_\cc)_\alpha=0$.  
Of course, $x_\alpha\geq 0$, since we assumed that $\xx\in\mathbb U_\cc$. 
Since the $g$-vectors of the $P_j[1]$ are negative standard basis vectors, $g^k_\alpha$ is non-positive, and 
it follows that 
the $m'_k$ must be non-positive.

Repeating the same argument
with the projectives rather than the shifted projectives, and taking into
account the fact that $(v_\cc)_{0j}>0$, we conclude that each of the $m'_k$
is bounded between $0$ and some negative number $B$.  Therefore $\mathbb U_\cc$ is bounded. \end{proof}

\begin{proof}[Proof of Theorem \ref{th-one}]
We begin by establishing that the vertices of 
$\mathbb U_\cc$ correspond bijectively to maximal 
compatible sets in $\mathcal I$, with the bijection sending
the vertex to the indices of the hyperplanes $p_\alpha=0$
on which it lies.  

Since $\mathbb E_c$ is $n$-dimensional, a vertex of 
$\mathbb U_c$ must lie on at least $n$ facets, and therefore
on at least $n$ hyperplanes of the form $p_\alpha=0$.  By
Lemma \ref{key}, the collection of hyperplanes 
corresponding to a vertex must be compatible.  The maximal
compatible sets of $\mathcal I$ are exactly the compatible
sets of size $n$.  Thus, every vertex of $\mathbb U_c$ 
corresponds to a cluster.

We now argue by induction that every cluster corresponds to a vertex of
$\mathbb U_\cc$.  Let us suppose that we have a collection of $n$ compatible indices from $\mathcal I$, such as 
$T=\{\alpha,\alpha_1,\dots,\alpha_{n-1}\}$, and another collection of $n$ compatible indices 
$T'=\{\alpha',\alpha_1,\dots,\alpha_{n-1}\}$.  Suppose that $\mathbb U_\cc$ has a vertex $q_T$ at the intersection of the facets corresponding to $T$.  Consider moving along the ray from $q_T$ where the facets corresponding to $\alpha_1,\dots,\alpha_{n-1}$ intersect.  Since $\mathbb U_\cc$ is bounded by Lemma \ref{bounded}, we must eventually hit another hyperplane bounding $\mathbb U_\cc$.
By Lemma \ref{key}, it must be a
hyperplane which is compatible with $\alpha_1,\dots,\alpha_{n-1}$; the only one is the
the hyperplane corresponding to $\alpha'$.  This intersection is a vertex of $\mathbb U_\cc$ corresponding to $T'$.  

We therefore know that $T'$ also corresponds to a vertex of $\mathbb U_\cc$. Since
the clusters are connected by mutations, every cluster corresponds to a vertex
of $\mathbb U_\cc$, and since by following edges, we only ever get to vertices
that correspond to clusters (never to an edge that goes to infinity, or
to a vertex that doesn't correspond to a cluster), all the vertices of
$\mathbb U_\cc$ correspond to clusters.  Further, the one-skeleton of $\mathbb U_\cc$ is
the cluster exchange graph.

We have therefore established point (3) of the theorem.
By Lemma \ref{key}, 
each face of $\mathbb U_c$ corresponds to a compatible set (necessarily all different),
and since
all the maximal compatible sets correspond to vertices of
$\mathbb U_\cc$, every compatible
set does correspond to a face. This establishes (1) and (2).  
(4) is straightforward.  Given a point in $\underline{y}\in \mathbb R^n$,
we can uniquely find an $\xx$ such that $\pi(\xx)=\underline{y}$ by inductively solving the $\cc$-deformed mesh relations, working from right to left through
the quiver for $\mathcal I$.  

Finally, we establish (5). Let $G$ be the $n\times |\mathcal I|$ matrix whose $(i,j)$-th column is the $g$-vector
$g(i,j)$.  Let $\sigma$ be the affine map from $\mathbb R^n$ to
$\mathbb E_\cc\subset V$ which is a section of $\pi$. This map is given by right multiplying by
$-G$ and adding $\vv_\cc$.  This induces a map from the 
tangent vectors of $V$ to tangent vectors of $\mathbb R^n$
(which we identify with $V$ and $\mathbb R^n$ respectively, via our fixed bases for each of them).
Linear algebra tells us that the map on tangent vectors is given by \emph{left} multiplying by $-G$.  Write $\ee_{ij}$ for the standard basis vector corresponding
to $(i,j)\in\mathcal I$.  Since $-\ee_{ij}$ is orthogonal to the hyperplane $p_{ij}=0$, and points away from $\mathbb U_\cc$, it image under left multiplication by $-G$ generates
the outer normal ray corresponding to this facet.  But
clearly $(-G)(-\ee_{ij})=g(i,j)$, as desired.  
\end{proof}

We close the section by commenting on the implications of 
Theoem \ref{th-one} for different assumptions on $\cc$.  

Clearly, if we assume that the constants $\cc$ are positive
rational numbers, Theorem \ref{th-one} still holds.  By 
continuity, it also holds for positive real numbers.  
Now consider the case that the some of the constants $\cc$
are zero.  

\begin{corollary}\label{cor1} If the constants $\cc$ are non-negative, then every vertex of $\mathbb U_\cc$ lies on 
a collection of coordinate hyperplanes which is the union of one or more maximal compatible sets in $\mathcal I$.  The facet normals to the facets of $\mathbb A_\cc$ are a subset of the $g$-vectors. \end{corollary}

\begin{proof}  We can imagine what happens if we begin with
$\cc'$ where all values are positive, and then deform 
gradually to a nearby vector $\cc$, where some values become zero.  What can happen is that some vertices can merge,
and some facets can collapse to something which is no longer codimension 1.  The results are as described in the statement of the corollary. \end{proof}

\section{Background on cluster algebras}\label{cluster-intro}

Let $Q$ be a quiver without loops or oriented two-cycles.  As already mentioned, there is an associated
cluster algebra $\mathcal A(Q)$, whose cluster variables are in natural bijection with the elements of $\mathcal I$, as we now explain.

A \emph{slice} in $\mathcal I$ is a subset $\mathcal T$ of
$\mathcal I$ such that:
\begin{itemize}
\item For each $j$ with $1\leq j\leq n$, there is exactly one element of
  $\mathcal T$ of the form $(i,j)$ for some $i$, and
\item if $j$ and $j'$ are adjacent vertices of $Q$, and $(i,j)$ and
  $(i',j')$ are the corresponding elements of $\mathcal T$, then
  $(i,j)$ and $(i',j')$ are adjacent in $\mathcal I$.
\end{itemize}

The collection of all $(0,j)$ form a slice, corresponding to the projective
representations, as does the collection $\mathcal I\setminus \mathcal I^+$,
corresponding to the shifted projectives.  We refer to $\{(0,j)\}$ as the
\emph{initial slice} and 
$\mathcal I\setminus \mathcal I^+$ as the \emph{final slice}.  

Note that any slice in $\mathcal I$ has at least one element $(i,j)$ such that
all the arrows between it and the other elements of the slice are oriented from
$(i,j)$.  We call such an $(i,j)$ a \emph{source} in the slice.  
Any slice other than the final slice has a source which
is in $\mathcal I^+$.

We assign the intial cluster variable $x_j$ to $(0,j)\in \mathcal I$.
Suppose we have a slice $\mathcal T$ which is not the final slice,
and the elements of $\mathcal T$ are
associated to the cluster variables of a cluster.  Let $(i,j)$ be a source
in the slice which is in $\mathcal I^+$.  We can therefore replace $(i,j)$ by
$(i+1,j)$ to obtain a new slice $\mathcal T'$.  
Mutate the cluster at the variable
corresponding to $(i,j)$, and associate the resulting cluster variable to
the vertex $(i+1,j)$. This provides a cluster associated to the slice
$\mathcal T'$.  As we proceed from the initial slice to the final slice,
we may well have choices of source.  However, these choices do not matter.
We have the following theorem, essentially from \cite{BMRRT}:

\begin{theorem} The above procedure results in a well-defined map from
  $\mathcal I$ to the cluster variables of $\mathcal A(Q)$, independent
  of choices.\end{theorem}

\begin{example}\label{ex-a3-d} In the setting of Example~\ref{ex-a3}, the result is the following:


    
\[\begin{tikzpicture}[xscale=1.3, yscale=1]
    \draw (-3,2.4) node (01)   {$x_1$};
    \draw (-3,0)   node (02)   {$x_3$};
    \draw (-2,1.2) node (03)   {$x_2$};
    \draw (-1,2.4) node (11)   {$\frac{x_2+1}{x_1}$};
    \draw (-1,0)   node (12)   {$\frac{x_2+1}{x_3}$};
    \draw (0,1.2)  node (13)   {$\frac{x_2^2+2x_2+x_1x_3+1}{x_1x_2x_3}$};
    \draw (1,2.4)  node (21)   {$\frac{x_1x_3+x_2+1}{x_2x_3}$};
    \draw (1,0)    node (22)   {$\frac{x_1x_3+x_2+1}{x_1x_2}$};
    \draw (2,1.2)  node (23)   {$\frac{x_1x_3+1}{x_2}$};
    \draw [->] (01) -- (03);
    \draw [->] (02) -- (03);
    \draw [->] (03) -- (11);
    \draw [->] (03) -- (12);
    \draw [->] (11) -- (13);
    \draw [->] (12) -- (13);
    \draw [->] (13) -- (21);
    \draw [->] (13) -- (22);
    \draw [->] (21) -- (23);
    \draw [->] (22) -- (23);
\end{tikzpicture} \]

\end{example} 

Let us emphasize that any slice in $\mathcal I$ corresponds to a cluster,
though there are also further clusters which are not slices.  
Note that if we restrict the quiver on $\mathcal I$ to the vertices of
the initial slice, we recover $Q$.  In fact, it is easy to show by 
induction that for any slice, the quiver associated to that cluster
is given by restricting the quiver on $\mathcal I$ to that slice.  

One way to construct a more complicated cluster algebra which is still governed by the same Dynkin combinatorics is to define a new \emph{ice quiver} $Q^\ice$ which is obtained from $Q$ by adding some vertices which are designated as \emph{frozen}.  The new vertices may be connected to the vertices of $Q$ in any way (provided that there are still no oriented two-cycles).
It is standard to assume that there are also no arrows between frozen vertices, but this is not actually important since such arrows play no role.  The cluster algebra associated to an ice quiver, $\mathcal A(Q^\ice)$ is the algebra generated by the initial cluster variables and all cluster variables obtained by all sequences of mutations at non-frozen vertices only.  Note that
there is a cluster variable associated to each frozen vertex.  To distinguish the frozen and unfrozen variables, we generally write $x_1,\dots,x_n$ for the initial unfrozen variables and $y_1,\dots,y_m$ for the frozen variables.  A refined version of the Laurent Phenomenon says that every cluster variable is contained in $\mathbb Z[x_1^\pm,\dots,x_n^\pm,y_1,\dots,y_m]$, see \cite[Theorem 3.3.6]{FWZ}. The unfrozen cluster variables still correspond to the elements of $\mathcal I$ and can be calculated in the same way, by starting with the initial variables associated to the initial slice and then carrying out mutations as before. 

There are two particular ice quivers obtained from $Q$ which are of special importance, one of which we will explain now.  
For each vertex $i$ of $Q$, add a frozen vertex $i'$ and an arrow from $i'$ to $i$.  The resulting quiver is called the \emph{framed quiver} of $Q$; we denote it $Q^\prin$.  To the new vertex $i'$ we associate the frozen cluster variable $y_i$.  
The cluster algebra associated to $Q^\prin$ is called the \emph{cluster algebra with principal coefficients}. We denote it $\mathcal A^\prin(Q)$. 

\begin{example}\label{ex-a3-e}
The framed quiver of $Q$ from Example ~\ref{ex-a3} is as follows, where the frozen vertices appear in squares.
\begin{center}
$Q^\prin: \xymatrix{1 \ar[r] & 2  & 3 \ar[l] \\
\boxed{{1^\prime}} \ar[u]  & \boxed{{2^\prime}} \ar[u]  & \boxed{{3^\prime}} \ar[u] }$
\end{center}
The cluster variables of the corresponding cluster algebra $\mathcal A^\prin(Q)$ are listed in Section \ref{prin}.
\end{example}

The significance of this choice of coefficients will be explained further in Section \ref{prin}.  Briefly, it turns out that from the cluster variables for the cluster algebra with principal coefficients, one can immediately calculate the cluster variables for any system of coefficients.  The cluster algebra with principal coefficients is also essential for defining $F$-polynomials, as we shall explain shortly.

An ice quiver with $n$ unfrozen vertices and $m$ frozen vertices can also be represented as an $(n+m)\times n$ matrix of integers, $\widetilde B_0$.  We associate the $n$ columns and the corresponding first $n$ rows to the $n$ unfrozen vertices of $Q$, and we associate the remaining rows to the frozen vertices.  The entry $(\widetilde B_0)_{ij}$ is the number of arrows from $i$ to $j$ minus the number of arrows from $j$ to $i$.  We see that the first $n$
rows of $\widetilde B_0$ are simply the matrix $B_0$ which we have already seen.  The matrix corresponding to 
$Q^\prin$ consists of $B_0$ with an $n\times n$ identity matrix below it.  

\begin{example}  For our running example, the extended matrix $\widetilde B_0$ corresponding to principal coefficients is the following:
$$\left[\begin{array}{ccc} 0&1&0\\ -1&0&-1\\ 0&1&0 \\ \hline 1&0&0\\0&1&0\\0&0&1\end{array}\right]$$
\end{example}

\section{Submodule polytopes and torsion classes}

For the duration of this section, we work with an arbitrary 
finite-dimensional algebra $A$, thought of as a the path algebra of a quiver $Q$ with 
relations, where the vertices of $Q$ are numbered 1 to $n$.  
We can therefore still consider dimension vectors of
such modules.  

For $X$ an $A$-module, we write $\Po_X$ for the polytope
in $\mathbb R^n$ which is the convex hull of the dimension vectors of 
subrepresentations of $X$. We call $\Po_X$ the \emph{submodule polytope}
of $X$.  
Submodule polytopes for representations of preprojective
algebras play an 
important r\^ole in the paper by Baumann, Kamnitzer and Tingley \cite{BKT} which we are essentially following in this
section.  

A \emph{torsion class} in $A$-mod is a full subcategory closed under extensions and quotients.  If $\mathcal T$ is a torsion 
class, any $A$-module $M$ has a unique largest submodule which is
contained in $\mathcal T$.  We call this the \emph{torsion part} of $M$ with 
respect to $\mathcal T$.  See \cite[Chapter VI]{ASS} for more
background on torsion classes.

The key fact about submodule polytopes is that, in order to find
the vertices of a submodule polytope, we do not need to consider
all submodules: it suffices to consider those submodules which are
torsion parts with respect to some torsion class.  
The following proposition is established in
\cite[Section 3]{BKT}.  Because we need rather less than is
established in \cite{BKT}, we give the simple proof here.

\begin{lemma}[{\cite{BKT}}] \label{BKTlemma} Let $M$ be an $A$-module.  Then $\Po_M$ equals the convex hull of the dimension vectors of the torsion parts
of $M$.  This amounts to saying that, for each vertex $\xx$ of $\Po_M$, there is some torsion class $\mathcal T$ with respect to which the dimension
vector of the torsion part of $M$ equals $\xx$. Further, the submodule 
of $M$ with this dimension vector is unique.  
\end{lemma}

\begin{proof} Let $\xx$ be a vertex of $\Po_M$.  It follows that there must be at least one submodule of $M$ whose dimension vector is $\xx$.   Let $N$ be such a submodule.  We will 
show that 
there is a torsion class $\mathcal T$ such that the torsion part of 
$M$ with respect to $\mathcal T$ is $N$.    

Choose a linear form $\theta$ on the space of dimension vectors such 
that the unique point on $\Po_M$ maximizing $\theta$ is $x$.  Define a 
torsion class by 
$$\mathcal T_\theta = \{ X \mid \theta ( \udim Y)\geq 0 \textrm{ for all quotients $Y$ of $X$}\}$$
It is not hard to establish $T_\theta$ is a torsion class (\cite[Proposition 3.1]{BKT}).  
Suppose $L$ is a proper submodule of $N$.  Since $L$ is also a submodule of 
$M$, we know that $\theta(\udim L)<\theta(\udim N)$, so 
$\theta(\udim N/L)>0$, and it follows that $N\in \mathcal T_\theta$. 
Let us write $N'$ for the torsion part of $M$ with respect to 
$\mathcal T_\theta$.  Since $N$ is a submodule of $M$ and is in $\mathcal T_\theta$, we know that $N$ must be a submodule of $N'$.  If it
were a proper submodule, then $\theta(\udim N')<\theta(\udim N)$, so $\theta(\udim N'/N)<0$, contradicting the assumption that $N'\in \mathcal T$.  Thus
$N'=N$.  

For the final statement, suppose that there were two distinct modules 
$N_1,N_2$ with dimension vector $x$.  The above argument shows that
both $N_1$ and $N_2$ are in $\mathcal T$, so the maximal torsion 
part of $M$ with respect to $\mathcal T$ is larger than either of them, which we have established is impossible.  
\end{proof}

\section{Newton polytopes of $F$-polynomials}

By definition, the $F$-polynomial $F_{\alpha}$ for $\alpha\in \mathcal I$
is obtained by taking $x_\alpha^\prin$, 
the cluster variable
associated to $\alpha$ in $\mathcal A^\prin(Q)$, and setting all the $x_i$ to 1.
It is therefore a
polynomial in $y_1,\dots,y_n$. 
It turns out that the cluster variable $x_\alpha^\prin$ can be recovered
from the $F$-polynomial (see \cite[Corollary 6.3]{CA4}), so no information has been 
lost, and at the same time, the $F$-polynomial turns out
to be convenient for another reason: for $(i,j)\in \mathcal I$ with
$i\geq 1$, the 
$F$-polynomial $F_{(i,j)}$ is the generating function for
the submodules of the representation $W_{(i-1,j)}$
in
the following sense:
$$F_{(i,j)}(y_1,\dots,y_n) = \sum_{\ee\leq \udim(W_{(i-1,j)})} \chi(\Gr_\ee(W_{(i-1,j)}) y^\ee$$
We refer to \cite[Eq. (1.6)]{DWZ} for the formula. Here $\ee=(e_1,\dots,e_n)$ is a dimension vector in $\mathbb Z_{\geq 0}^n$,
$\Gr_\ee(X)$ means the quiver Grassmannian of subrepresentations of $X$ whose dimension vector is $\ee$,
and $\chi$ is the Euler characteristic.  We write $y^\ee$ for 
$y_1^{e_1}y_2^{e_2}\dots y_n^{e_n}$.

Let $f\in \mathbb Z[y_1,\dots,y_n]$ be a polynomial.
Let $P$ be the subset of $\mathbb Z^n$ such that 
$f$ can be written as $$f=\sum_{\pp\in P} f_\pp y^\pp$$
with all $f_\pp$ non-zero integers.  That is to say $P$ is the
collection of the the $n$-tuples corresponding to
exponents of terms which appear in $f$.  
The \emph{Newton polytope} of $f$ is then the 
convex hull of the points in $P$.  

Newton polytopes of $F$-polynomials have been studied 
by Brodsky and Stump \cite{BS}.  They give a description in
type $A_n$ and a uniform conjecture for all finite types.
Subsequent to the first appearance of the present paper, this conjecture
  was proved in \cite{JLS}, relying in part on our results.

Newton polytopes of cluster variables have also been
studied, by Sherman and Zelevinsky in rank 2 \cite{SZ},
by Cerulli Irelli for $\widetilde A_2$ \cite{CI}, and
by Kalman in type $A_n$ \cite{K}.  Note that, by \cite[Corollary 6.3]{CA4},
the Newton polytope of a cluster variable is an
affine transformation of the corresponding $F$-polynomial,
so the two questions are quite close.

For $(i,j)\in \mathcal I^+$, let $\ee_{ij}$ denote the 
standard basis vector in $\mathbb R^{\mathcal I^+}$
which has a 1 in position $(i,j)$
and zeros elsewhere. 

\begin{theorem} \label{th-two} $\mathbb A_{\ee_{ij}}$
  is the Newton polytope of $F_{(i+1,j)}$.
\end{theorem}

Subsequent to the appearance of the first version of the present paper, this result has been extended to the non-simply laced case \cite{AHL}, by using a folding argument to reduce to the simply laced case, for which they rely on this result.
Before we prove the theorem, we will state and prove a key lemma, and then a proposition.  

\begin{lemma} Let $T$ be a tilting object in the additive hull of $\{W_\alpha\mid \alpha\in\mathcal I\}$.  Let $\mathcal T$ be the corresponding torsion class in $\rep Q^\op$, consisting of all quotients of sums of summands of $T$ which are contained in $\rep Q^\op$. For $M\in \rep Q^\op$, let $tM$ denote the torsion part of $M$ with respect to the torsion class $\mathcal T$.
Let $\qq_T$ be the vertex of $\mathbb U_\cc$ at which the coordinates corresponding to summands of $T$ are zero.  
Then $\pi(\qq_T)=\udim(tM)$.  
\end{lemma}

\begin{example}
To see examples of this lemma, we can revisit Example \ref{ex-a3-c}.
The first vertex listed corresponds to $T_1=kQ^\op$.  The torsion class $\mathcal T_1=\rep Q^\op$ is the full category of representations, so $t_1M_\cc=M_\cc$, and $\pi(q_{T_1})=(3,4,3)$, which is the dimension of $M_\cc$.

Looking at the second vertex listed, we see that $\mathcal T_2$ consists of direct sums of all indecomposables except $S_3$.  The corresponding torsion part $t_2M_\cc$ therefore has dimension $(3,4,2)=\pi(q_{T_2})$.  (Recall that in the definition of $\pi$, the order in which the final slice of coordinates appear is determined by their corresponding $g$-vectors, which is why, in type $A_3$, $\pi$ is in effect reading the final slice of coordinates from bottom to top.)  
\end{example}

\begin{proof}

As before, let $G$ be the $n\times |\mathcal I|$ matrix whose $(i,j)$-th 
column consists of the $g$-vector $g(i,j)$.  

Let the summands of $T$ be $T_1,\dots,T_n$.  
Let us write $G|_T$
for the $n\times n$ matrix formed by taking the columns
of $G$ corresponding to $T_1,\dots,T_n$.  
We can interpret $G|_T$ as the change-of-basis matrix from the basis 
$\udim T_1,\dots,\udim T_n$ to the basis $\udim P_1,\dots,\udim P_n$.
$G|_T$ is therefore invertible, with inverse given by the inverse change of basis.  It follows that
$(G|_T)^{-1}G$ is a matrix whose restriction to the columns 
corresponding to summands of $T$ is an identity matrix.  

The point in $\mathbb E_\cc$ which has zeros in the columns corresponding to the $T_i$ is
therefore $$\vv_\cc-[\dim\Hom(T_1,M_\cc),\dots,\dim\Hom(T_n,M_\cc)](G|_T)^{-1}G$$

Note that since
$T$ is a tilting object, each indecomposable projective module $P_i$ 
admits a coresolution 
$$P_i \rightarrow T_{i0} \rightarrow T_{i1}\rightarrow P_i[1],$$
where $T_{i0}$ and $T_{i1}$ are in add $T$.  The entries in the $i$-th
column of $(G|_T)^{-1}$ encode the signed multiplicity of 
$T_1,\dots,T_n$ in this coresolution of $P_i$.  

Note that the final $n$ coordinates of $\vv_\cc$ are zero.  The $i$-th coordinate of $\mathbb A_\cc$ is therefore $\dim\Hom(T_{i0},M_\cc)-\dim\Hom(T_{i1},M_\cc)$.  

From the coresolution of $P_i$, we obtain the following commutative diagram, with the rows exact:

$$ \begin{tikzpicture}[xscale=1.4,yscale=1.2]
\node (y) at (-3.5,0) {0};
\node (x) at (-2,0) {$\Hom(T_{i1},M_\cc)$};
\node (a) at (0,0) {$\Hom(T_{i0},M_\cc)$};
\node (b) at (2,0) {$\Hom(P_i,M_\cc)$};
\node (c) at (4,0) {$\Ext^1(T_{i1},M_\cc)$};
\node (d) at (0,-1) {$\Hom(T_{i0},tM_\cc)$};
\node (e) at (2,-1) {$\Hom(P_i,tM_\cc)$};
\node (f) at (4,-1) {$\Ext^1(T_{i1},tM_\cc)$};
\node (z) at (-3.5,-1) {0};
\node (w) at (-2,-1) {$\Hom(T_{i1},tM_\cc)$};
\draw[-stealth](a)--(b);
\draw[-stealth](b)--(c);
\draw[-stealth](e)--(f);
\draw[-stealth](d)--(e);
\draw[-stealth](d)--(a);
\draw[-stealth](e)--(b);
\draw[-stealth](f)--(c);
\draw[-stealth](y)--(x);
\draw[-stealth](x)--(a);
\draw[-stealth](z)--(w);
\draw[-stealth](w)--(d);
\draw[-stealth](w)--(x);
\end{tikzpicture}$$

The zeros on the lefthand end follow from the fact that
  $\Hom(P_i,M_\cc[-1])=0=\Hom(P_i,tM_\cc[-1])$. Further,
$\Ext^1(T_{i1},tM_\cc)=0$ because $T_{i1}$ is 
$\Ext$-projective in $\mathcal T$ while $tM_\cc$ is in $\mathcal T$.  
Therefore the map from $\Hom(T_{i0},tM_\cc)$ to $\Hom(P_i,tM_\cc)$ is 
surjective.  

The first three vertical maps are injective because they are induced from the inclusion of $tM_\cc$ into $M_\cc$.  

Any map from a torsion module to $M_\cc$ necessarily lands in the torsion part
of $M_\cc$, so factors through $tM_\cc$.  This means that the first two
vertical arrows are also surjective.  

Our goal is to understand the image of $\Hom(T_{i0}, M_\cc)$ 
inside $\Hom(P_i,M_\cc)$; by what we have already shown, it equals $\Hom(P_i,tM_\cc)$; in other words, the dimension of this image is the dimension of $tM_\cc$ at vertex $i$, as desired. 
\end{proof}

From the previous lemma, the following proposition is almost immediate:

\begin{proposition}\label{p1} $\mathbb A_\cc=\Po_{M_\cc}$.
\end{proposition}

\begin{proof} 
Thanks to Corollary \ref{cor1}, we know that the vertices of $\mathbb U_\cc$ are the set of points $\qq_T$ for $T$ a tilting object in the additive hull of the $W_\alpha$, with $\alpha\in \mathcal I$.  (Note that it is of course posible that $\qq_T=\qq_{T'}$ for two distinct tilting objects $T$ and $T'$.) By the previous lemma, $\pi(\qq_T)$ is the dimension vector of the torsion part of $M_\cc$ with respect to the corresponding torsion class.  All torsion classes are of this form, so $\mathbb A_\cc$ is the convex hull of the dimension vectors of all possible torsion parts of $M_\cc$.  
Lemma \ref{BKTlemma} now tells us that $\mathbb A_\cc$ is therefore 
the submodule polytope of $M_\cc$, as desired.    
\end{proof}

We can now prove Theorem \ref{th-two}.

\begin{proof}[Proof of Theorem \ref{th-two}]
  Since we are interested in $\mathbb A_{\ee_{ij}}$, we set $\cc=\ee_{ij}$, $M_\cc=W_{ij}$.
  The quiver Grassmannian $\Gr_\ee(W_{ij})$ is empty if $\ee$ is not the dimension
  vector of a submodule.  Thus, the Newton polytope of $F_{(i+1,j)}$ is
  contained in the convex hull of the dimension vectors of submodules of
  $W_{ij}$, which we have established in Proposition \ref{p1} is
  $\mathbb A_{\ee_{ij}}$.  It remains to check that the vertices of $\mathbb A_{\ee_{ij}}$
  correspond to quiver Grassmannians with non-zero Euler characteristics.  Lemma \ref{BKTlemma} tells us that for each
  vertex of $\mathbb A_{\ee_{ij}}$, there is a unique submodule of 
  the appropriate dimension vector.  The quiver Grassmannian is
  therefore a single point, and the Euler characteristic of a single point is 1.
  \end{proof}  
  
\section{The use of principal coefficients}\label{prin}

Let $Q^\ice$ be an ice quiver, whose unfrozen part is $Q$.  
We will explain, following \cite{CA4}, how the cluster variables of $\mathcal A(Q^\ice)$ can be calculated directly from those of $\mathcal A^\prin(Q)$, rather than via mutation.  




Let $f(x_1,\dots,x_n,y_1,\dots,y_n)$ be a cluster variable in 
$A^\prin(Q)$.  By $F(y_1,\dots,y_n)$ we denote the associated $F$-polynomial, which is obtained by setting 
$x_1=\dots=x_n=1$ in $f$.

Let $z_1,\dots,z_m$ be the coefficients corresponding to the frozen vertices of $Q^\ice$ (equivalently, these correspond to rows 
$n+1$ to $n+m$ of the matrix $\widetilde B_0$.)  
Define
$$\tilde y_i=\prod_{j=1}^m z_j^{b_{n+j,i}}$$

We write $F^\trop(\tilde y_1,\dots,\tilde y_n)$ for the tropical evaluation of $F$ at $\tilde y_1,\dots,\tilde y_n$.  This is the monomial such that the power of $z_i$ that appears in it is the minimum over all terms of $F$ of the power of $z_i$ in that term.  (This is the gcd of the monomials that appear.)

Then \cite[Theorem 3.7]{CA4} says that the cluster variable in $\mathcal A(Q)$ corresponding to $f$
is equal to 
$$f(x_1,\dots,x_n,\tilde y_1,\dots, \tilde y_n)/F^\trop(\tilde y_1,\dots,\tilde y_n).$$

\begin{example} Let us consider the following example, 
\begin{center}
$Q^\ice: \xymatrix{1 \ar[r] & 2  & 3 \ar[l] \ar[dl]\\
& \fbox{4} \ar[ul] & }$
\end{center} 
with vertex \boxed{4} frozen, and associated to the variable $z$.  
The corresponding $B$-matrix is:
$$\left[\begin{array}{ccc} 0&1&0\\ -1&0&-1\\ 0&1&0 \\ \hline 1&0&-1\end{array}\right]$$

Using the labeling of the vertices as in Example~\ref{ex-a3}, the cluster variables as well as the cluster variables with coefficients associated to every vertex are as following:

$$\begin{array}{SlSlSl}
      & \text{{Cluster variables for $Q^\ice$}} & \text{{Cluster variables with principal coefficients}} \\
$(0,1)$ & $x_1$  & $x_1$\\
$(0,2)$ & $x_2$  & $x_2$\\
$(0,3)$ & $x_3$  & $x_3$\\
$(1,1)$ &$\dfrac{x_2+z}{x_1}$ & $\dfrac{x_2 + y_1}{x_1}$\\
$(1,2)$ & $\dfrac{x_1x_3z+x_2^2z+x_2z^2+x_2+z}{x_1x_2x_3}$ & $\dfrac{x_1x_3y_1y_2y_3+x_2^2+x_2y_1+x_2y_3+y_1y_3}{x_1x_2x_3}$\\
$(1,3)$ &$\dfrac{x_2z+1}{x_3}$ & $\dfrac{x_2 + y_3}{x_3}$\\
$(2,1)$ & $\dfrac{x_1x_3z+x_2+z}{x_1x_2}$ & $\dfrac{x_1x_3y_1y_2+x_2+y_1}{x_1x_2}$ \\
$(2,2)$ & $\dfrac{x_1x_3+x_2z+1}{x_2x_3}$ & $\dfrac{x_1x_3y_2y_3+x_2+y_3}{x_2x_3}$ \\
$(2,3)$ & $\dfrac{x_1x_3+1}{x_2}$ & $\dfrac{x_1x_3y_2+1}{x_2}$\\
\end{array}$$



By the definition of $\tilde y_i$, in this example we have $\tilde y_1=z$, while $\tilde y_2=1$ and $\tilde y_3=z^{-1}$.
Substituting them in the polynomials listed above, we obtain the polynomials $f(x_1,\dots,x_n,\tilde y_1,\dots, \tilde y_n)$. Moreover, the monomials $F^\trop(\tilde y_1,\dots,\tilde y_n)$ associated to every polynomial $f$ are as follows:

$$\begin{array}{SlSlSl}
      & ${f(x_1,\dots,x_n,\tilde y_1,\dots, \tilde y_n)}$ & ${F^\trop(\tilde y_1,\dots,\tilde y_n)}$ \\
$(0,1)$ & $x_1$  & 1\\
$(0,2)$ & $x_2$  & 1\\
$(0,3)$ & $x_3$  & 1\\
$(1,1)$ &$\dfrac{x_2+z}{x_1}$ & 1\\
$(1,2)$ & $\dfrac{x_1x_3+x_2^2+x_2z+x_2z^{-1}+1}{x_1x_2x_3}$ & $z^{-1}$\\
$(1,3)$ &$\dfrac{x_2+z^{-1}}{x_3}$ & $z^{-1}$\\
$(2,1)$ & $\dfrac{x_1x_3z+x_2+z}{x_1x_2}$ & 1 \\
$(2,2)$ & $\dfrac{x_1x_3z^{-1}+x_2+z^{-1}}{x_2x_3}$ & $z^{-1}$ \\
$(2,3)$ & $\dfrac{x_1x_3+1}{x_2}$ & 1\\
\end{array}$$


Calculating $f(x_1,\dots,x_n,\tilde y_1,\dots, \tilde y_n)/F^\trop(\tilde y_1,\dots,\tilde y_n)$ and comparing with the cluster variables above, we observe that the theorem holds in this case.  \end{example}

\section{Universal coefficients}

As mentioned in Section \ref{cluster-intro}, there are two choices of coefficients which are particularly 
interesting.  One is the principal coefficients which we discussed in the previous section. The other is \emph{universal coefficients}.

In fact, there are two closely related notions: universal coefficients, introduced by Fomin and Zelevinsky \cite{CA4}, and universal geometric coefficients, introduced by Reading \cite{Reading}.  In both cases,
the goal is a cluster algebra with sufficiently general coefficients that it will admit a ring homomorphism (with certain good properties) to the cluster algebra defined for any other choice of coefficients.  In the finite type case, the two definitions yield the same system of 
coefficients, see \cite{Reading}.  We will not need
any properties of universal coefficients, so we do not
give the precise definitions.  

Reading ~\cite[Theorem 10.12]{Reading} proves $\widetilde B_0$ provides universal coefficients if the coefficient rows of the extended exchange matrix are the $g$-vectors $B_0^T$, where $T$ indicates transposition. 

Since, for us, $B_0$ is skew-symmetric, $B_0^T=-B_0$.
Thus, the desired coefficient rows are the $g$-vectors for $Q^\op$.  Reading the quiver for $\mathcal I$ from right to left instead of left to right, we see that the $g$-vectors for $Q^\op$ are simply the negatives of the $g$-vectors for $Q$.  Thus, we shall be interested in the setting where we add a row to the exchange matrix $B_0$ for each element of $\mathcal I$, with the row corresponding to $(i,j)$ being given by $-g(i,j)$.  
We will call the corresponding algebra $\mathcal A^\univ(Q)$.

\pagebreak
\begin{example} Consider again Example~\ref{ex-a3}. The extended exchange matrix $\widetilde{B}_0$ is

$$\left[ \begin{array}{ccc}
 0 & 1 & 0\\
 -1 & 0 & -1\\
0 & 1 & 0\\ \hline
-1 & 0 & 0\\
 0 &-1 & 0\\
 0 & 0 &-1\\
 1 & -1 &0\\
 1 & -1 &1\\
 0 & -1 &1\\
  0 & 0 & 1\\
 0 & 1 & 0\\
 1 & 0 & 0
\end{array}
\right]$$

The additional rows in $\widetilde{B}_0$ define the behaviour of 
$z_{0,1},\dots,z_{2,3}$, where $z_{ij}$ is the frozen variable corresponding to the row whose entries are $-g(i,j)$.    

The cluster variables with universal coefficients are computed below. 
$$\begin{array}{SlSl}
      & \text{Cluster variables with universal coefficients} \\
$(0,1)$ & $x_1$  \\
$(0,2)$ & $x_2$  \\
$(0,3)$ & $x_3$  \\
$(1,1)$ &$\dfrac{x_2z_{0,1} + z_{1,1}z_{1,2}z_{2,3}}{x_1}$ \\
$(1,2)$ & $(x_2^2z_{0,1}z_{0,2}z_{0,3} + x_2z_{0,1}z_{0,2}z_{1,2}z_{1,3}z_{2,1} + x_2z_{0,2}z_{0,3}z_{1,1}z_{1,2}z_{2,3}$ \\ &$+ z_{0,2}z_{1,1}z_{1,2}^2z_{1,3}z_{2,1}z_{2,3} + x_1x_3z_{1,2}z_{2,1}z_{2,2}z_{2,3})/x_1x_2x_3$ \\
$(1,3)$ &$\dfrac{x_2z_{0,3} + z_{1,2}z_{1,3}z_{2,1}}{x_3}$ \\
$(2,1)$ &$\dfrac{x_2z_{0,2}z_{0,3}z_{1,1} + z_{0,2}z_{1,1}z_{1,2}z_{1,3}z_{2,1} + x_1x_3z_{2,1}z_{2,2}}{x_2x_3}$ \\
$(2,2)$ &$\dfrac{z_{0,2}z_{1,1}z_{1,2}z_{1,3} + x_1x_3z_{2,2}}{x_2}$ \\
$(2,3)$ & $\dfrac{x_2z_{0,1}z_{0,2}z_{1,3} + z_{0,2}z_{1,1}z_{1,2}z_{1,3}z_{2,3} + x_1x_3z_{2,2}z_{2,3}}{x_1x_2}$ \\
\end{array}$$
\end{example}


Define the universal $F$-polynomial $F^\univ_{(i,j)}$
to be the polynomial
obtained starting from the cluster variable in position $(i,j)$ in $\mathcal A^\univ(Q)$, and setting the initial
cluster variables to one. 

\begin{theorem} $\mathbb U_{\ee_{(i,j)}}$ is
the Newton polytope of $F^\univ_{(i+1,j)}$.
\end{theorem}

\begin{example}
Consider, for example, the universal $F$-polynomial $F^\univ_{(1,1)}$.  According to the above calculation, it is $z_{0,1}+z_{1,1}z_{1,2}z_{2,3}$.  Thus, its Newton polytope is the line segment from $(1,0,0,0,0,0,0,0,0)$ to $(0,0,0,1,1,0,0,0,1)$.  To find the vertices of $\mathbb A_{\ee_{(0,1)}}$, we find solutions to the $\ee_{(0,1)}$-deformed mesh relations which are all non-negative and have at least three zeros.  The results are the following, as expected.  

\[\begin{tikzpicture}[xscale=.9, yscale=.75]
    \draw (-3,2.4) node (01) {$1$};
    \draw (-3,0) node (02)   {$0$};
    \draw (-2,1.2) node (03) {$0$};
        \draw (-2,2.4) node[red] {$1$};
        \draw (-2,0) node[red] {$0$};
        \draw (-1,1.2) node[red] {$0$};
    \draw (-1,2.4) node (11) {$0$};
    \draw (-1,0) node (12)   {$0$};
    \draw (0,1.2) node (13)  {$0$};
        \draw (0,2.4) node[red] {$0$};
        \draw (0,0) node[red] {$0$};
        \draw (1,1.2) node[red] {$0$};
    \draw (1,2.4) node (21)  {$0$};
    \draw (1,0) node  (22)   {$0$};
    \draw (2,1.2) node (23)  {$0$};
    \draw [->] (01) -- (03);
    \draw [->] (02) -- (03);
    \draw [->] (03) -- (11);
    \draw [->] (03) -- (12);
    \draw [->] (11) -- (13);
    \draw [->] (12) -- (13);
    \draw [->] (13) -- (21);
    \draw [->] (13) -- (22);
    \draw [->] (21) -- (23);
    \draw [->] (22) -- (23);
\end{tikzpicture} \qquad
\begin{tikzpicture}[xscale=.9, yscale=.75]
    \draw (-3,2.4) node (01) {$0$};
    \draw (-3,0) node (02)   {$0$};
    \draw (-2,1.2) node (03) {$0$};
        \draw (-2,2.4) node[red] {$1$};
        \draw (-2,0) node[red] {$0$};
        \draw (-1,1.2) node[red] {$0$};
    \draw (-1,2.4) node (11) {$1$};
    \draw (-1,0) node (12)   {$0$};
    \draw (0,1.2) node (13)  {$1$};
        \draw (0,2.4) node[red] {$0$};
        \draw (0,0) node[red] {$0$};
        \draw (1,1.2) node[red] {$0$};
    \draw (1,2.4) node (21)  {$0$};
    \draw (1,0) node  (22)   {$1$};
    \draw (2,1.2) node (23)  {$0$};
    \draw [->] (01) -- (03);
    \draw [->] (02) -- (03);
    \draw [->] (03) -- (11);
    \draw [->] (03) -- (12);
    \draw [->] (11) -- (13);
    \draw [->] (12) -- (13);
    \draw [->] (13) -- (21);
    \draw [->] (13) -- (22);
    \draw [->] (21) -- (23);
    \draw [->] (22) -- (23);
\end{tikzpicture} \]








\end{example}

\begin{proof}  We will use the strategy explained in Section \ref{prin} to calculate $F^\univ_{(i+1,j)}$ on the basis of $F_{(i+1,j)}$.  Define $\tilde y_1,\dots, \tilde y_n$ as in Section \ref{prin} with respect to the matrix $\widetilde B_0$ as defined above.  Then $$F^\univ_{(i+1,j)}=\frac{F_{(i+1,j)}(\tilde y_1,\dots,\tilde y_n)}{F^\trop_{(i+1,j)}(\tilde y_1,\dots,\tilde y_n)}.$$

Let $\xx$ be a vertex of $\mathbb A_\cc$.  It corresponds to a term $y^\xx$ in $F_{(i+1,j)}$.  
There
is a unique element $\tilde \xx \in \mathbb E_\uzero$ such that
$\pi(\tilde \xx)=\xx$.  This is precisely the exponent vector of the result of substituting 
$\tilde y_i$ for $y_i$ in the monomial $y^\xx$.  

The element of $\mathbb E_\cc$ which projects onto $\xx$ is exactly
$\tilde\xx + \vv_\cc$. This, then, is the vertex of $\mathbb U_\cc$ 
corresponding to $\xx$.
What remains to be verified is that 
$\vv_\cc$ equals negative the exponent of
$F^\trop_{(i+1,j)}(\tilde y_1,\dots,\tilde y_n)$.  

By definition, $F^\trop_{(i+1,j)}(\tilde y_1,\dots,\tilde y_n)$ is the least common multiple of all the terms in $F_{(i+1,j)}(\tilde y_1,\dots,\tilde y_n)$.
Subtracting it from the Newton polytope of $F_{(i+1,j)}(\tilde y_1,\dots,\tilde y_n)$ translates the polytope so that for each $z_{ij}$, the 
minimal power that appears is zero.  We know that $\mathbb U_\cc$ 
also has the property that the minimum value which any coordinate
takes on within $\mathbb U_\cc$ is zero.  Thus $\vv_\cc=-F^\trop_{(i+1,j)}(\tilde y_1,\dots,\tilde y_n)$, as desired.   
\end{proof}

\begin{example} We look at $F_{(1,1)}^\univ$ in our running example.
\begin{eqnarray*}
\tilde y_1&=&z_{1,1}z_{1,2}z_{2,3}/z_{0,1}\\
\tilde y_2&=&z_{2,2}/z_{0,2}z_{1,1}z_{1,2}z_{1,3}\\
\tilde y_3&=&z_{1,2}z_{1,3}z_{2,1}/z_{0,3}
\end{eqnarray*}
Now $$F_{(1,1)}(\tilde y_1,\tilde y_2,\tilde y_3)=\frac{z_{0,1} + z_{1,1}z_{1,2}z_{2,3}}{z_{0,1}}, \quad F^\trop_{(1,1)}(\tilde y_1,\tilde y_2,\tilde y_3)=z^{-1}_{0,1}$$
and we indeed obtain $$F^\univ_{(1,1)}=F_{(1,1)}(\tilde y_1,\tilde y_2, \tilde y_3)/F^\trop_{(1,1)}(\tilde y_1,\tilde y_2,\tilde y_3).$$
\end{example}

\section{The nef cone of the toric variety associated to the $g$-vector fan is simplicial}

\subsection{Brief reminder on toric varieties}
Our main reference for toric varieties is \cite{CLS}. We do not give
specific references for the basic facts which are to be found in the first few
chapters of that book.

Let $N$ be a free abelian group of rank $n$. Write $M=\Hom(N,\mathbb Z)$ for
its dual. We write $\langle \cdot,\cdot\rangle$ for the duality pairing from
$M\times N$ to $\mathbb Z$.
We write $N_{\mathbb R}$ for $N \otimes_{\mathbb Z}\mathbb R$, and
in general use a subscript $\mathbb R$ to denote tensoring by
$\mathbb R$.

A cone in a real vector space is a semigroup closed under multiplication by
non-negative reals.
A strongly convex, rational, polyhedral cone in $N_{\mathbb R}$ is a cone
which is
generated by a
finite collection of vectors from $N$, all lying in a proper half-space.
A fan in $N_{\mathbb R}$ is a collection of strongly convex,
rational, polyhedral cones
such that the
intersection of any two cones is necessarily a face of each.
(The examples to
which we will apply this theory are the outer normal fan to the generalized associahedra $\mathbb A_\c$ which we have
constructed, which are particular instances of $g$-vector fans of finite type
cluster algebras. They are indeed fans in the above sense.)

Let $\Sigma$ be a fan in $N_{\mathbb R}$.
For simplicity of exposition, we will assume that the $n$-dimensional
cones of
$\Sigma$ cover $\mathbb R^n$, since this is true in particular
for
the outer normal fans which we are considering.
We write $\Sigma^i$ for the $i$-dimensional cones of $\Sigma$.

Associated to a fan $\Sigma$, there is a normal toric variety $X_\Sigma$. 
The field of rational functions on $X_\Sigma$ is the fraction field of
the group ring $\mathbb C[M]$. For $\m\in M$, we write $\chi^\m$ for the
corresponding function field element.

There is a torus $T\simeq(\mathbb C^*)^n$ acting on $X_\Sigma$. There is a bijection
between cones of $\Sigma$ and $T$-orbits in $X_\Sigma$. We write $\mathcal O_\sigma$ for the orbit corresponding to the cone $\sigma \in \Sigma$.
The dimension of $\mathcal O_\sigma$ is
$n$ minus the dimension of the span of $\sigma$.

A divisor on a normal variety is a formal $\mathbb Z$-linear
combination of irreducible codimension one subvarieties.
On a toric variety, we are
particularly interested in those divisors which are torus invariant.
For $\rho\in\Sigma^1$, we define $D_\rho$ to be the closure of $\mathcal O_\rho$.
This is an irreducible codimension one subvariety which is torus-invariant.
The torus-invariant divisors of
$X_\Sigma$ are $\Div(X_\Sigma)=\bigoplus_{\rho\in \Sigma^1} \mathbb Z [D_\rho]$.

Given a normal variety $X$ and an element of the function field of $X$, say
$f$, there is an associated divisor, $\div(f)$. Informally, it consists of the
zero locus of $f$ minus the locus where $f$ blows up. We shall shortly define
this notion precisely in the setting of toric varieties.
A divisor $D$ on a normal variety $X$
is called a Cartier divisor if there exists an
open cover $\{U_i\}$ of $X$ such that $D|_{U_i}$ is principal for each $i\in I$,
that is to say, there exists an element $f_i$ of the function field of $U_i$ such
that $\div(f_i)$ equals the restriction of $D$ to
$U_i$.

In the case of toric varieties, we are interested in torus-invariant
Cartier divisors. It turns out that there is a canonical choice of
open cover which works for any torus-invariant Cartier divisor.
For each maximal cone $\sigma$, let $X_\sigma$ denote the
toric variety associated to the fan consisting of $\sigma$ and its faces.
$X_\sigma$ is open in $X_\Sigma$, and $X_\Sigma$ is covered by the varieties
$X_\sigma$. Any torus-invariant Cartier divisor on $X_\Sigma$ is given by
a collection of functions, one on each $X_\sigma$. In fact, we can take the function on $X_\sigma$ to be of the form $\chi^{\m_\sigma}$ for some
$\m_\sigma\in M$. We call the collection $\{\m_\sigma\}_{\sigma\in\Sigma^n}$ the local data of the
Cartier divisor.
For each ray $\rho\in\Sigma^1$, let $\u_\rho$ be the first lattice point along the ray $\rho$.
The multiplicity of $D_\rho$ in the divisor corresponding to the local data
$ \{\m_\sigma\}_{\sigma\in\Sigma^n}$ is given
by $-\langle \m_\sigma,\u_\rho\rangle$ for $\sigma$ any cone of
$\Sigma$ containing $\rho$.

A collection $\{\m_\sigma\}_{\sigma\in\Sigma^n}$
is not necessarily the local data of any Cartier
divisor. The following condition provides a necessary and sufficient condition
to verify that it is.

\begin{lemma}[{\cite[Theorem 4.2.8, Exercise 4.2.3]{CLS}}]
  \label{cart}
  The collection $\{\m_\sigma\}_{\sigma\in\Sigma^n}$ forms the
  local data for a Cartier divisor if and only if
$\langle \m_\sigma, \u_\rho\rangle$ is equal for all maximal cones $\sigma$
  containing the ray $\rho$.\end{lemma}

As we have already remarked, $-\langle \m_\sigma,\u_\rho\rangle$ is
the multiplicity of $D_\rho$ in the divisor; the condition for
$\{\m_\sigma\}$ to be local data
amounts to saying that the formula for the multiplicity of $D_\rho$ does not depend on which $\sigma$ is used, among those containing $\rho$.

We write $\CDiv(X_\Sigma)$ for the torus invariant Cartier divisors on $X_\Sigma$. We write
$\Div_0(X_\Sigma)$ for the principal divisors on $X_\Sigma$. These are the
Cartier divisors for which all $m_\sigma$ are equal. The Picard group of
$X_\Sigma$, denoted $\Pic(X_\Sigma)$, is defined to be $\CDiv(X_\Sigma)/\Div_0(X_\Sigma)$. Under our assumptions
on the fan $\Sigma$, $\CDiv(X_\Sigma)$ is a free abelian group whose rank
is $|\Sigma^1|-n$ \cite[Theorem 4.2.1]{CLS}.

Let $D$ be a Cartier divisor on a normal variety $X$ and let $C$ 
a complete curve in $X$.
We write $D\cdot C$ for the intersection product of $D$ and $C$.
We will not define it in full generality, but the following can be taken
as a definition in the toric setting.

\begin{definition}\label{inter}
On the toric variety $X_\Sigma$, let $D$ be a torus-invariant
Cartier divisor, which is thus given by a collection of local data
$\{\m_\sigma\}_{\sigma\in\Sigma^n}$. Let $C$ be a complete, irreducible
torus-invariant curve in $X_\Sigma$; it is therefore the closure of $\mathcal O_\tau$
for some codimension one $\tau$ in $\Sigma$. The cone $\tau$ 
separates two maximal cones $\sigma$ and $\sigma'$. Let $u$ be an
element of $\sigma'$ which maps to a generator of $N/N_\tau$.
(Here $N_\tau$ is the lattice generated by $\tau$, so $N/N_\tau$ is isomorphic to $\mathbb Z$.)
Then
$$D\cdot C=\langle \m_\sigma - \m_{\sigma'},\u\rangle$$
\end{definition}
We take this as the definition of $D\cdot C$; see also
\cite[Proposition 6.3.8]{CLS}, which shows that this definition agrees with the
definition for general varieties. 

A Cartier divisor on a normal variety $X$ is called nef (``numerically
effective'') if $D\cdot C\geq 0$ for every irreducible complete curve $C$ in $X$. In the toric
case, $D$ is nef if and only if $D\cdot C\geq 0$ for every irreducible
torus-invariant complete curve, i.e., for those curves which are the closure of
$\mathcal O_\tau$ for some $\tau\in\Sigma^{n-1}$ \cite[Theorem 6.3.12]{CLS}.
Thus, the definition of $D\cdot C$ which we
have given above is sufficient to determine, for any Cartier divisor on $X_\Sigma$, whether or not
it is nef. 

A Cartier divisor $D$ on a normal variety is said to be  
numerically equivalent to zero iff $D\cdot C=0$
for all irreducible complete curves $C$; two Cartier divisors are
numerically equivalent if their difference is numerically equivalent to zero.
The nef cone of a variety is defined in the vector space of its Cartier
divisors modulo numerical equivalence, and tensored by $\mathbb R$.
For a toric variety, a torus-invariant
Cartier divisor is numerically equivalent to
zero if and only if it is principal \cite[Proposition 6.3.15]{CLS}.
Thus, we can view $\Nef(X_\Sigma)$ as contained in
$\Pic(X_\Sigma)_{\mathbb R}$. Then, $\Nef(X_\Sigma)$ is the cone
generated by the classes of the nef Cartier divisors in $\Pic(X_\Sigma)_{\mathbb R}$.

Associated to a torus-invariant Cartier divisor $D=\sum_{\rho\in\Sigma^1} a_\rho [D_\rho]$ on $X_\Sigma$, there is a polytope defined by  $$P_D=\{\m\in M_{\mathbb R}
\mid \la \m,\u_\rho\ra \geq -a_\rho \textrm { for all } \rho\in \Sigma^1\}.$$
See \cite[(6.1.1)]{CLS}.

Combining \cite[Theorems 6.1.7 and 6.3.12]{CLS}, we obtain the following
useful criterion for a Cartier divisor's being nef:
\begin{proposition}\label{nefprop}
  A Cartier divisor $D$ on $X_\Sigma$ with local data $\{\m_\sigma\}_{\sigma\in\Sigma^n}$ is nef if and only if $\m_\sigma\in P_D$ for all $\sigma \in \Sigma^n$.
\end{proposition}

A cone in a real vector space is said to be simplicial if its number of generating rays is equal
to the dimension of its span.

We can now state the theorem which we seek to prove about the
toric variety associated to the $g$-vector fan.

\begin{theorem} \label{simp} Let $\Sigma$ be the outer normal fan of $\mathbb A_\c$ (or equivalently the 
  $g$-vector fan corresponding to
  a Dynkin quiver). The nef cone of the toric variety $X_\Sigma$ is
  simplicial. \end{theorem}

\subsection{Combinatorics of $\Pic(X_\Sigma)$}

Let $N=\mathbb Z^n$ be a rank $n$ free abelian group with a fixed basis
$\b_1,\dots,\b_n$,
and let $M=\Hom(N,\mathbb Z)$, equipped with the dual basis $\b_1^*,\dots,\b_n^*$.
Let $\Sigma$ be the $g$-vector fan realized in $N_{\mathbb R}$ with respect
to the basis $\{\b_i\}$, and let $X_\Sigma$ be the corresponding toric
variety. The irreducible divisors correspond to rays of $\Sigma$, which
themselves correspond to elements $(i,j)\in \mathcal I$. We will write
$D_{ij}$ for the divisor corresponding to the ray $g(i,j)$.

We identify $M_{\mathbb R}$ with
the vector space $\mathbb R^n$ which is the image of the projection
$\pi$, identifying $\b_i^*$ with the
standard basis of $\mathbb R^n$. 

We will now construct a bijection between $\mathbb Z^{|\mathcal I^+|}$ and
$\Pic(X_\Sigma)$, the
Cartier divisors of $X_\Sigma$ up to linear equivalence.

Fix $\c=(c_{ij})_{(i,j)\in \mathcal I^+}$, with all $c_{ij}\in \mathbb Z$.
  Unlike earlier in the paper, we do not assume that the entries of
  $\c$ are non-negative. Nonetheless, as earlier, we get a well-defined
  affine subspace $\mathbb E_\c$ of $V$. There is a section of $\pi$,
  which we denote
  $\i_\c$, sending $\mathbb R^n$ to $\mathbb E_\c$. 

  The definition we gave of $\v_\c$ in Section \ref{section-three} does not
  make sense any longer, since some of the $c_{ij}$ are negative. However,
  it can be extended to the more general setting in the following way.

  Define
  a matrix $A$, whose rows are indexed by $\mathcal I$ and whose
  columns are indexed by $\mathcal I^+$, and such that the $(i,j),(k,l)$ entry
  is $\dim\Hom(M_{ij},M_{kl})$. If we order the rows and columns of $A$ in the same way, respecting the left-to-right order of the Auslander--Reiten quiver of $Q$ and keeping the elements of
  $\mathcal I\setminus \mathcal I^+$ for last, we obtain a matrix with 1's on
  the diagonal and zeros below it, with the last $n$ rows consisting
  entirely of zeros.

  We then define $\v=A\c$. Clearly, this recovers the definition of
  Section \ref{section-three} when the $c_{ij}$ are non-negative. The
  proof that $\v\in\mathbb E_\c$ goes through without essential alterations.
  
  For future use, let us write
  $A'$ for the submatrix of $A$ consisting of the rows indexed by
  elements of $\mathcal I^+$. Since it is upper triangular with 1's on
  the diagonal, it is invertible in $GL_{|\mathcal I^+|}(\mathbb Z)$. 

\begin{example} \label{exa}
Consider the case that $Q=\xymatrix{1 \ar[r] & 2}$.

    The elements of $\mathcal I$ are as follows:
    
      \[\begin{tikzpicture}[xscale=1.3, yscale=1]
    \draw (-3,2.4) node (01) {$(0,1)$};
    \draw (-2,1.2) node (03) {$(0,2)$};
    \draw (-1,2.4) node (11) {$(1,1)$};
    \draw (0,1.2) node (13)  {$(1,2)$};
    \draw (1,2.4) node (21)  {$(2,1)$};
    \draw[-stealth] (01)--(03);
    \draw[-stealth](03)--(11);
    \draw[-stealth](11)--(13);
    \draw[-stealth](13)--(21);
      \end{tikzpicture}\]

      In this case, there is only one ordering of $\mathcal I$ consistent
      with the left-to-right ordering of the Auslander--Reiten quiver,
      namely $(0,1),(0,2),(1,1),(1,2),(2,1)$. With respect to this ordering
      the matrix $A$ is given by
      $$ A=\left[\begin{array}{ccc} 1&1&0\\0&1&1\\0&0&1\\0&0&0\\0&0&0\end{array}\right]$$
        The matrix $A'$ consists of the top three rows of $A$.
\end{example}
  
  Let $G$ be the $n\times |\mathcal I|$ matrix whose $(i,j)$-th column is the
  $g$-vector corresponding to $(i,j)\in\mathcal I$. 
  As already noted in the proof of point (5) of Theorem \ref{th-one},
  $$i_c(\y)=-\y\cdot G + \v_\c.$$ 
  
  Pullback along $i_\c$ defines a map from functions on $V$ to functions
  on $\mathbb R^n=M_{\mathbb R}$. In particular, we can consider the pullbacks of the
  coordinate functions $i_c^*(p_{ij})$. From the formula for $i_\c(\y)$,
  we obtain:
  \begin{equation}\label{eq3}i_\c^*(p_{ij})(\y)=-\y\cdot g({i,j}) + v_{ij}.
    \end{equation}
 
  The zero locus of
  $i_\c^*(p_{ij})$ is an affine hyperplane in $\mathbb R^n = M_{\mathbb R}$.
  
  The maximal cones in $X_\Sigma$ correspond to maximal compatible sets
  in $\mathcal I$. For a maximal cone $\sigma$, define $-\m_\sigma(\c)$ to be the
  intersection of the zero loci of the pullbacks by $i_\c$
  of the coordinate functions
  corresponding to the rays of $\sigma$. Since the corresponding $g$-vectors
  are linearly independent, this intersection is a well-defined point.

  \begin{lemma} \label{a} The collection $\{\m_\sigma(\c)\}$ provides local data for a
    Cartier divisor on $X_\Sigma$. This Cartier divisor is
  $D(\c)=\sum_{(i,j)\in\mathcal I^+} v_{ij}[D_{ij}]$.\end{lemma}

  \begin{proof}
    As defined, the points $\m_\sigma(\c)$ lie in $M_{\mathbb R}$. We must first check
    that they are in fact elements of $M$. This follows from the fact that
    the $g$-vectors corresponding to the rays of $\Sigma$ form a basis
    for $\mathbb Z^n$.

    Now we check that the $\m_\sigma(\c)$
    satisfy the necessary condition to be local
    data of a Cartier divisor, as recalled in Lemma \ref{cart}.
    Let $\rho$ be a ray of $\Sigma$ corresponding to
    $(i,j)\in \mathcal I$. We must check that
    the value of $\langle \m_\sigma,\u_\rho\rangle$
    is independent of the choice of $\sigma$
    containing $\rho$. 
    This is so because by construction all the points
    $-\m_\sigma$ lie on the hyperplane in $M_\mathbb R$ where $i_c^*(p_{ij})=0$.
    Now $\u_\rho=g(i,j)$, and
    $i_c^*(p_{ij})(-\m_\sigma)=\la \m_\sigma,g(i,j)\ra + v_{ij}$, so
    $\la \m_\sigma,\u_\rho\ra = -v_{ij}$, independent of $\sigma$, as desired.
\end{proof}

  We now establish the following converse to the previous lemma:

  \begin{lemma} \label{b} Up to linear equivalence, any Cartier divisor on $X_\Sigma$ is given by local data
    $\{m_\sigma(\c)\}$ for some $\c$. \end{lemma}

  \begin{proof} Suppose we have a Cartier divisor $D=\sum_{(i,j)\in\mathcal I}
    v_{ij}[D_{ij}]$. There is a principal divisor whose coefficients with
    respect to the rays corresponding to $(i,j)\in \mathcal I\setminus
    \mathcal I^+$ take any integer values, so, by subtracting it from
    $D$, we may assume that $v_{ij}=0$ for $(i,j)\not\in \mathcal I^+$.
    Consider the vector $\v$ which is the $|\mathcal I^+|$-tuple
    consisting of the entries $v_{ij}$ for $(i,j)\in \mathcal I^+$.
    To construct $D$ as a Cartier divisor, we must show that there exists
    a $|\mathcal I^+|$-tuple of integers $\c$ such that 
    $\v=A'\c$. Since, as we have already commented, $A'$ is invertible,
    we find that $\c=(A')^{-1}(\v)$ gives us the necessary $\c$.
  \end{proof}

  From the two previous lemmas, we deduce:
  \begin{proposition} \label{iso} There is an isomorphism of abelian groups between
    $\mathbb Z^{\mathcal I^+}$ and $\Pic(X_\Sigma)$, sending $\c$ to $D(\c)$.
  \end{proposition}

  \begin{proof} Lemma \ref{a} establishes the existence of the desired
    map, which is clearly a morphism of groups,
    and Lemma \ref{b} shows that it is surjective.
    By \cite[Theorem 4.2.1]{CLS}, which we have already cited,
    $\Pic(X_\Sigma)$ is a free abelian group of rank $|\mathcal I|-n=|\mathcal I^+|$.
    Since $\mathbb Z^{\mathcal I^+}$ and $\Pic(X_\sigma)$ are
    free abelian groups of the same rank, a surjective map from one to the
    other must be an isomorphism.
    \end{proof}

  \begin{example}
We continue the setting of Example \ref{exa}.
    
      Let us set $c_{01}=2$, $c_{02}=1$, $c_{11}=1$. The figure below shows
      $\mathbb A_\c$ and $P_{D(\c)}$.

\definecolor{ffzzqq}{rgb}{1,0.6,0}
\definecolor{xdxdff}{rgb}{0.49019607843137253,0.49019607843137253,1}

\begin{tikzpicture}[line cap=round,line join=round,>=triangle 45,x=1cm,y=1cm]
\begin{axis}[
x=1.35cm,y=1.35cm,
axis lines=middle,
xmin=-4,
xmax=4,
ymin=-2,
ymax=2,
xtick={-7,-6,...,8},
ytick={-6,-5,...,5},]
\clip(-7.914920738377092,-6.756204103012107) rectangle (8.687945243638579,5.796772504847912);
\fill[line width=2pt,color=xdxdff,fill=xdxdff,fill opacity=0.42] (0,0) -- (0,1) -- (1,2) -- (3,2) -- (3,0) -- cycle;
\fill[line width=2pt,color=ffzzqq,fill=ffzzqq,fill opacity=0.41] (-3,0) -- (0,0) -- (0,-1) -- (-1,-2) -- (-3,-2) -- cycle;
\draw (1.3,1.2) node[anchor=north west] {$\mathbf{\mathbb{A}_{\underline{c}}}$};
\draw (-1.9,-0.7) node[anchor=north west] {$P_{D(\underline{c})}$};
\end{axis}
\end{tikzpicture}
      
 \end{example}
    
  \subsection{Proof of Theorem \ref{simp}}

We are now almost ready to prove Theorem \ref{simp}.
  
  \begin{lemma} If all the entries of $\c$ are non-negative, then
    $D(\c)$ is nef. 
  \end{lemma}

  \begin{proof} Let $D(\c)=\sum_{(i,j)\in\mathcal I} v_{ij}[D_{ij}]$.
    By Proposition \ref{nefprop}, it suffices to show that,
    provided the entries of $\c$ are non-negative integers, then 
    $\m_\sigma(\c)\in P_{D(\c)}$ for all $\sigma\in\Sigma^n$. 

    $\mathbb A_\c$ is essentially by definition
    the region cut out by the inequalities
    $i^*_\c(p_{ij})(\y) \geq 0$. By (\ref{eq3}), this
    is equivalent to $-\y\cdot g(i,j) + \v_{ij}\geq 0$. Thus,
    $\mathbb A_\c$ is cut out by the inequalities $-\y\cdot g(i,j) \geq -v_{ij}$.
    This says that $\mathbb A_\c=- P_{D(\c)}$.

    Now, by Theorem \ref{th-one}, we know that, assuming all the entries of
    $\c$ are non-negative, the points $-\m_\sigma(\c)$ are
    the vertices of $\mathbb A_\c$, and thus the points $\m_\sigma(\c)$ lie in
    $P_{D(\c)}$.
    \end{proof}
  
  \begin{lemma} If $\c$ has a negative entry, then $D(\c)$ is not nef.
  \end{lemma}

  \begin{proof} Recall that, as defined in Section \ref{cluster-intro}, a slice of the AR quiver of $Q$ is a choice, for each
    $1\leq i \leq n$ of a vertex $(i,s(i))$ such that if vertices $i$ and
      $i'$ of $Q$ are adjacent, then $(i,s(i))$ and $(i+1,s(i+1))$ are adjacent in the AR quiver.
Any slice is a compatible set. 
    
Suppose that $c_{kl}<0$. Choose two slices, $s$ and $s'$, so that
$s$ contains $(k,l)$, and $s'$ contains $(k,l+1)$, while the other vertices
in the two slices are the same. Write $\sigma$ and $\sigma'$ for the
corresponding cones of $\Sigma$. Since the two cones differ only in one ray,
they share a common codimension 1 face, $\tau$. Write $C$ for the
corresponding curve. We will show that $D(\c)\cdot C < 0$, showing that
$D(\c)$ is not nef.

We know that
$i_\c(\m_{\sigma}(\c))_{i,s(i)}=0$ for $1\leq i \leq n$. This implies that $i_\c(\m_{\sigma}(\c))_{k,l+1}=-c_{kl}$,
by the $\c$-deformed mesh relation, 
while $i_\c(\m_{\sigma'}(\c))_{i,s'(i)}=0$, so $i_\c(\m_{\sigma'}(\c))_{k,l+1}=0$.

To apply Definition \ref{inter}, we may take $\u=g(k,l+1)$. We find that $D(\c)\cdot C =
\langle \m_\sigma(\c)-\m_{\sigma'}(\c),g(k,l+1)\rangle =
i_c^*(p_{k,l+1})(-\m_\sigma) - i_c^*(p_{k+1,l})(-\m_{\sigma'}) = c_{kl}<0$
\end{proof}

  \begin{proof}[Proof of Theorem \ref{simp}]
    Proposition \ref{iso} establishes an isomorphism of abelian groups from
    $\mathbb Z^{|\mathcal I|}$ and $\Pic(X_\Sigma)$, and thus a linear transformation from $\mathbb R^{|\mathcal I|}$ to $\Pic(X_\Sigma)_{\mathbb R}$. By the previous two lemmas, the  cone
    $\Nef(X_\Sigma)$ is the image in $\Pic(X_\Sigma)_{\mathbb R}$
    of the positive orthant in $\mathbb R^{|\mathcal I|}$, and is thus
    simplicial.
\end{proof}

\subsection*{Acknowledgements}
This work was initiated in the LaCIM representation theory working group, and benefitted from discussion with the other members of the group, including Aram Dermenjian, Patrick Labelle, and Franco Saliola.   KM thanks Anna Felikson for her hospitality and fruitful discussion of the early stage of this work during his visit to Durham, UK. HT would like to thank Nima Arkani-Hamed, Frédéric Chapoton, Giovanni Cerulli Irelli, Giulio Salvatori, and Christian Stump for helpful conversations and comments.


\begin{thebibliography}{ABHY}
\bibitem{ABHY} N. Arkani-Hamed, Y. Bai, S. He, and G. Yan,
\emph{Scattering forms and the positive geometry of kinematics, color and the worldsheet}, Journal of High Energy Physics 2018, article number 96.

\bibitem{AHL} N. Arkani-Hamed, S. He, and T. Lam,
  \emph{Cluster configuration spaces of finite type}, arXiv:2005.11419.

\bibitem{ASS} I. Assem, D. Simson, and A. Skowro\'nski,  
\emph{Elements of the representation theory of associative algebras}, Volume 1, Cambridge University Press, Cambridge, 2006.

\bibitem{BKT} P. Baumann, J. Kamnitzer, and P. Tingley,  \emph{Affine Mirkovi\'c-Vilonen polytopes.} Publ. Math. Inst. Hautes \'Etudes Sci. 120 (2014), 113--205.

\bibitem{BM} V. Bazier-Matte, \emph{Combinatoire des algèbres amassées}.
  Ph.D. thesis, Université du Québec à Montréal, 2020.
  
\bibitem{BS} S. Brodsky and C. Stump,
\emph{Towards a uniform subword complex description of acyclic
finite type cluster algebras.}  Algebraic Combinatorics 1 (2018), no.~4, 545--572.

\bibitem{BMRRT} A. B. Buan, B. Marsh, M. Reineke, I. Reiten, and G. Todorov,
\emph{Tilting theory and cluster combinatorics}, Advances in Mathematics 204 (2), 572--618, 2006. 

\bibitem{CSZ} C. Ceballos, F. Santos, and G. M. Ziegler, 
\emph{Many non-equivalent realizations of the associahedron}, Combinatorica 35 (5), 513--551, 2015.

\bibitem{CI} G. Cerulli Irelli, \emph{Cluster algebras
of type $A_2^{(1)}$},
Algebras and Represention Theory 15, no. 5, 977--1021, 2012.

\bibitem{CFZ} F. Chapoton, S. Fomin, and A. Zelevinsky,
\emph{Polytopal realizations of generalized associahedra}, Canadian Mathematical Bulletin 45 (4), 537--566, 2002.

\bibitem{CLS} D. Cox, J. Little, and H. Schenck. \emph{Toric varieties}.
  American Mathematical Society, Providence, RI, 2011.


\bibitem{DWZ} H. Derksen, J. Weyman,  and A. Zelevinsky,
\emph{Quivers with potentials and their representations {II}: applications to cluster algebras}, Journal of the American Mathematical Society 23, 749--790, 2010.
\bibitem{F1} J. Fei, \emph{Combinatorics of $F$-polynomials}. arXiv:1901.10151v3.

\bibitem{F2} \bysame, \emph{Tropical $F$-polynomials and general presentations},
  arXiv:1911.10513v2.
  
\bibitem{FWZ} S. Fomin, L. Williams, and A. Zelevinsky,
\emph{Introduction to cluster algebras: {C}hapters 1--3},
arXiv:1608.05735v2.

\bibitem{FZ} S. Fomin and A. Zelevinsky,
\emph{{$Y$}-systems and generalized associahedra}, Annals of Mathematics. Second Series 158 (3), 977--1018, 2003.

\bibitem{CA4} \bysame,
\emph{Cluster algebras {IV}: {C}oefficients}, Compositio Mathematica 143 (1), 112--164, 2007.


  
\bibitem{HLT} C. Hohlweg, C. Lange, and H. Thomas,
\emph{Permutahedra and generalized associahedra}, Advances in Mathematics 226 (1), 608--640, 2011.

\bibitem{HPS} C. Hohlweg, V. Pilaud, and S. Stella,
\emph{Polytopal realizations of finite type $g$-vector fans}, arXiv:1703.09551v2.
\bibitem{JLS} D. Jahn, R. Löwe, and C. Stump,
  \emph{Minkowski decompositions for generalized associahedra of finite type},
  arXiv:2005.14065.

\bibitem{K} A. Kalman, \emph{Newton polytopes of cluster 
variables of type $A_n$}.  arXiv:1310.0555.

\bibitem{Lee} C. Lee, \emph{The associahedron and triangulations of the $n$-gon}. European J. Combinatorics 10 (6), 551--560, 1989.

\bibitem{Loday} J. L. Loday,
\emph{Realization of the Stasheff polytope}, Archiv der Mathematik 83 (3), 267--278, 2004.

\bibitem{MRZ} B. Marsh, M. Reineke, and A. Zelevinsky,
\emph{Generalized associahedra via quiver representations}, Transactions of the American Mathematical Society 355 (10), 4171--4186, 2003.

\bibitem{PPPP} A. Padrol, Y. Palu, V. Pilaud, and P.-G. Plamondon,
  \emph{Associahedra for finite type cluster algebras and minimal relations between g-vectors}, arXiv:1906.06861v2.

\bibitem{Postnikov} A. Postnikov,
\emph{Permutohedra, associahedra, and beyond}, International Mathematics Research Notices. IMRN 2009 (6), 1026--1106, 2009.

\bibitem{Reading} N. Reading, 
\emph{Universal geometric cluster algebras}, Mathematische Zeitschrift 277, 499--547, 2014.

\bibitem{RSS}  G. Rote, F. Santos, and I. Streinu,
\emph{Expansive motions and the polytope of pointed pseudo-triangulations}, Discrete and computational geometry, Algorithms Combin. 25, 699--736, Springer, Berlin, 2003.

\bibitem{SZ} P. Sherman and A. Zelevinsky,
\emph{Positivity and canonical bases in rank 2 cluster algebras of finite and affine types},  
Moscow Mathematics Journal 4 (4), 947--974, 982, 2004. 

\bibitem{SS} S. Shnider and S. Sternberg,
\emph{Quantum groups}, Graduate Texts in Mathematical Physics, II, International Press, Cambridge, 1993.

\bibitem{Sta} J. Stasheff, \emph{Homotopy associativity of $H$-spaces} I, Trans. Amer. Math. Soc. 138, 275--292, 1963.


\end{thebibliography}
\end{document}